\documentclass[12pt]{amsart}

\usepackage[colorlinks=true, urlcolor=black, citecolor=black, linkcolor=black, hyperfootnotes=true]{hyperref}
\usepackage{cleveref,stmaryrd,amssymb,graphicx,tikz}

\numberwithin{equation}{section}

\makeatletter

\DeclareRobustCommand{\btleft}{\mathop{\mathpalette\btlr@\blacktriangleleft}}
\DeclareRobustCommand{\btright}{\mathop{\mathpalette\btlr@\blacktriangleright}}

\newcommand{\btlr@}[2]{%
  \begingroup
  \sbox\z@{$\m@th#1\triangleright$}%
  \sbox\tw@{\resizebox{1\wd\z@}{1\ht\z@}{\raisebox{\depth}{$\m@th#1\mkern-1mu#2$}}}%
  \ht\tw@=\ht\z@ \dp\tw@=\dp\z@ \wd\tw@=\wd\z@
  \copy\tw@
  \endgroup
}

\theoremstyle{plain}
\newtheorem{theorem}{Theorem}[section]

\newtheorem{corollary}[theorem]{\bf Corollary}

\newtheorem{proposition}[theorem]{\bf Proposition}

\theoremstyle{definition}
\newtheorem{remark}[theorem]{\bf Remark}
\newtheorem*{remark*}{\bf Remark}
\newtheorem{definition}[theorem]{\bf Definition}

\author{Tristan Bice}
\address{Institute of Mathematics of the Czech Academy of Sciences, \v{Z}itn\'a 25, Prague}
\email{bice@math.cas.cz}


\begin{document}

\title{Dependent Types Simplified}

\begin{abstract}
We present two logical systems based on dependent types that are comparable to ZFC, both in terms of simplicity and having natural set theoretic interpretations.  Our perspective is that of a mathematician trained in classical logic, but nevertheless we hope this paper might go some way to bridging the cultural divide between type theorists coming from computer science.
\end{abstract}

\maketitle
\tableofcontents

\section*{Introduction}

\subsection*{Background}

Type theory has its origins in Russell and Whitehead's Principia \cite{RussellWhitehead1910} from the early 20th century, an ambitious project to set all mathematics of the day on a firm foundation.  Broadly speaking it achieved this goal, showing how a significant amount of mathematics could, at least in principal, be expressed in their formalism.  Nevertheless, the complexity of its ramified type system meant there was little enthusiasm for using it in practice, especially once simpler foundational systems like Zermelo-Fraenkel set theory become available.

The lambda-calculus underlying modern type theories also had its origin in Church's attempts in \cite{Church1932} to give an alternative foundation for mathematics.  While his foundational system was found to be flawed, part of the syntax was proposed in \cite{Church1936} as a formal system for computations.  For some time this also had little uptake in the general mathematical community, but for a rather different reason, namely that no precise meaning was given to the formalism.  It was only after Scott gave it some proper semantics that it started to gain more acceptance (see \cite[\S5]{Mislove1998}).  A similar story pertains to intuitionistic logic, for example, which was generally viewed with skepticism until Kripke semantics came along (see \cite{Kripke1965}).

Thus if history is any guide, to be taken seriously as a foundation for mathematics, a logical system must satisfy two basic criteria -- it must be both simple and meaningful.  The classic example of such a system is ZFC (Zermelo-Fraenkel set theory with the axiom of choice -- see \cite{Kunen2011}) which has indeed served as a solid foundation for mathematics for over a century now.  The language of ZFC is certainly quite simple, just the usual first order logic with a single relation symbol $\in$.  The logical symbols have their usual meaning, while $\in$ is naturally interpreted as the membership relation on some universe of sets.  The deductive system of ZFC is also just the usual one for first order logic with only two inference rules, modus ponens and universal generalisation.  The only slight complexity arises in the axioms, consisting of several logical axioms together with 9 more axioms specific to ZFC.

On the other hand, logical systems based on dependent types usually fail the basic criteria on both counts, being neither simple nor meaningful.  Indeed, a typical paper on dependent type theory will start with several pages of ad hoc inference rules with little or no effort devoted to explaining the meaning or motivation behind them.  Even those papers that do try to give the formalism some proper semantics do so only after the inference rules have already been fixed.  They also often make some unusual choices, e.g.~allowing the interpretations to be affected by assumptions, not just the underlying variable assignment.

It is no surprise then that such systems have not been seen as genuine contenders for foundational systems, despite the best efforts of their proponents.  This may be slowly changing with the advent of several proof assistants based on dependent types that have been gaining traction, at least in the computer science community.  However, even those mathematicians sympathetic to the idea of automated proof checking tend to balk at the type systems being used as the foundational logic.

Another problematic issue is the habit of type theorists to propose systems that are not even consistent.  From the very beginning, early versions of the Principia were inconsistent (see \cite{Linsky2002}), as were the original systems proposed by Church, Curry and Martin-L\"of (see \cite{KleeneRosser1935}, \cite{Coquand1994} and \cite{Hurkens1995}).  Right up to the present day, `bugs' are regularly being found in the kernels of the most popular proof assistants based on dependent types (see \cite{Carneiro2024}).  In contrast, no inconsistency has been found in ZFC since it was finalised over a century ago.  Thus an added benefit of a type system with rigorous set based semantics is that its consistency is put on a par with that of ZFC, the most well studied and well tested logical system in history.  Indeed, had the proponents of various inconsistent systems given more thought to their semantics from the outset, they could have quickly seen where the inconsistencies would likely arise and modified their systems accordingly.

That said, despite their chequered past, we believe dependent types do have potential as a foundational system if developed with just a little more care.  This is the goal of the present paper -- to show that systems of dependent types can be created which are simple enough to be compared with ZFC and which can also be meaningfully interpreted in an elementary but precise way in some universe of sets.  Even if these particular systems do not see any practical use, we hope they might at least serve as a bridge of understanding between the two communities in question -- mathematicians can gain a quick understanding of dependent types, while computer scientists can see some type systems that would be more appealing to mathematicians.

\subsection*{Related Work}

We would be remiss if we did not also mention simple type theories, which can indeed be sufficiently simple and meaningful that they could rightly be viewed as satisfying our criteria.  Why they have not received widespread recognition as potential foundational systems we can only speculate.  Suffice to say that the interest seems to have shifted towards dependent type theories, perhaps because of some nice features they possess, e.g.~a uniform treatment of terms and types and the compatibility of universal quantification with the `propositions as types' philosophy.

We have also neglected to mention the significant body of work devoted to categorical semantics and the more sophisticated type systems designed to handle them, like homotopy type theory.  As a foundation for higher category theory, some kind of homotopy type theory may well prove to be the best framework, as will no doubt be borne out by continued research in this area.  However, a foundational system for mathematics as a whole should first and foremost have some natural set based semantics, simple enough to be understood by any mathematician, not just those with expertise in category theory.  Of course any semantics beyond sets, category theoretic or otherwise, is definitely welcome, it is just not the primary goal of such a foundational system.

\section*{Preliminaries}

We follow the usual mathematical practice of working in some na\"ive universe of sets, making use of standard set theoretic notation and terminology throughout.  This can all be formalised in ZFC but this is not the main point -- much of what we say is simple enough that it would be valid in any reasonable formal theory of sets.  If particular axioms of ZFC (as detailed in \cite[\S I.3]{Kunen2011}, for example) or large cardinals become relevant, we will explicitly point this out.

As usual, the \emph{power set} of a set $S$ will be denoted by
\[\mathcal{P}(S)=\{R\mathrel{|}R\subseteq S\}.\]
Ordered pairs will be denoted by angle brackets, e.g.~$\langle r,s\rangle$.  The \emph{product} of any sets $R$ and $S$ is thus given by
\[R\times S=\{\langle r,s\rangle\mathrel{|}r\in R\text{ and }s\in S\}.\]
We often use the standard infix notation to denote an ordered pair $\langle r,d\rangle$ being a member of a set $S$, i.e.
\[r\mathrel{S}d\qquad\Leftrightarrow\qquad\langle r,d\rangle\in S.\]
The \emph{domain} and \emph{range} of a set $S$ are then given by
\begin{align}
    \tag{Domain}\mathrm{dom}(S)&=\{d\mathrel{|}\exists r\,(r\mathrel{S}d)\}.\\
    \tag{Range}\mathrm{ran}(S)&=\{r\mathrel{|}\exists d\,(r\mathrel{S}d)\}.
\end{align}
The \emph{inverse} of a set $S$ is given by
\[\tag{Inverse}S^{-1}=\{\langle d,r\rangle\mathrel{|} r\mathrel{S}d\}.\]

The \emph{image} of a set $S$ under a set $R$ is given by
\[\tag{Image}R[S]=\{r\mathrel{|}\exists s\in S\,(r\mathrel{R}s)\}.\]
In particular, the image $R[\{s\}]$ of a singleton set $\{s\}$ will be abbreviated as $R\{s\}$.  The \emph{application} of a set $R$ to a set $s$ is given by
\[\tag{Application}R(s)=\bigcup R\{s\}=\bigcup\{r\mathrel{|}\langle r,s\rangle\in R\}.\]
In particular, if $s\notin\mathrm{dom}(R)$ then $R(s)=R\{s\}=\emptyset$.  The \emph{composition} of a set $R$ with a set $S$ is given by
\[\tag{Composition}R\circ S=\{\langle r,s\rangle\mathrel{|}R^{-1}\{r\}\cap S\{s\}\neq\emptyset\}.\]
Equivalently, for any $r$ and $s$,
\[r\mathrel{R\circ S}s\qquad\Leftrightarrow\qquad\exists t\,(r\mathrel{R}t\mathrel{S}s).\]
The \emph{transitive closure} of $R$ is defined to be the smallest set $R^\mathsf{t}$ containing $R$ such that $R^\mathsf{t}\circ R^\mathsf{t}\subseteq R^\mathsf{t}$, i.e.
\[\tag{Transitive Closure}R^\mathsf{t}=R\cup(R\circ R)\cup(R\circ R\circ R)\cup\ldots.\]

A \emph{binary relation} $R$ is a set of ordered pairs.  Equivalently, $R$ is a binary relation if and only if $R\subseteq\mathrm{ran}(R)\times\mathrm{dom}(R)$.  A \emph{function} $F$ is a binary relation where, for each $x\in\mathrm{dom}(F)$, there is a unique $y\in\mathrm{ran}(F)$ with $y\mathrel{F}x$.  Equivalently, $F$ is a function if and only if $F=\{\langle F(x),x\rangle\mathrel{|}x\in\mathrm{dom}(F)\}$.

\begin{remark*}
When $F$ is a function, we immediately see that
\[y=F(x)\qquad\Leftrightarrow\qquad y\mathrel{F}x.\]
Indeed, this is usually taken as the definition of application, i.e.~ $F(x)$ is defined to be the unique $y$ satisfying $y\mathrel{F}x$, but only when $F$ is a function and $x$ lies in its domain.  However, we find it convenient to define applications for arbitrary sets as in \cite{Aczel1999} in order that our interpretations below can be defined on arbitrary terms.

Also note that for us a function with domain $D$ and range $R$ is subset of $R\times D$ not $D\times R$, as often seen in the literature.  We adopt this convention so that the notion of composition above agrees with the standard notions for both relations and functions.
\end{remark*}

The set of all functions from a set $X$ to a set $Y$ will be denoted by
\[Y^X=\{F\subseteq Y\times X:F\text{ is a function with }\mathrm{dom}(F)=X\}.\] 
More generally, given $F$ with $\mathrm{dom}(F)=X$, the set of functions $f$ on $X$ such that $f(x)\in F(x)$, for all $x\in X$, will be denoted by
\[\prod F=\prod_{x\in X}F(x)=\{f\in\big(\bigcup_{x\in X}F(x)\big)^X\mathrel{|}f(x)\in F(x),\text{ for all }x\in X\}.\]
This is called the \emph{dependent product} defined by $F$.

A \emph{string} $S$ in some alphabet $\mathsf{A}$ is an element of $\mathsf{A}^{<\omega}=\bigcup_{n\in\omega}\mathsf{A}^n$, i.e.~ a function from some finite ordinal $n=\{0,\ldots,n-1\}$ to elements of $\mathsf{A}$.  We call $n$ the \emph{length} of $S$.  We concatenate strings as usual so if $R$ and $S$ are strings of length $m$ and $n$ respectively then $RS$ denotes the string of length $m+n$ defined by
\[RS(k)=\begin{cases}R(k)&\text{if }k<m\\S(k-m)&\text{if }k\geq m.\end{cases}\]
Likewise, given sets of strings $\mathsf{R},\mathsf{S}\subseteq\mathsf{A}^{<\omega}$, we let
\[\mathsf{R}\mathsf{S}=\{RS\mathrel{|}R\in\mathsf{R}\text{ and }S\in\mathsf{S}\}.\]
We also identify symbols $a\in\mathsf{A}$ with the corresponding strings of length $1$, i.e.~where $a(0)=a$.  Thus if $a,b,c\in\mathsf{A}$ then $abc$ denotes the string of length $3$ such that $abc(0)=a$, $abc(1)=b$ and $abc(2)=c$.

We will introduce our type systems in the same way classical logical systems are usually introduced.  First we specify the language, i.e.~the alphabet $\mathsf{A}$ and the strings $\mathsf{C},\mathsf{V},\mathsf{T}\subseteq\mathsf{A}^{<\omega}$ that we take as the \emph{constants}, \emph{variables} and \emph{terms} of our language.  These terms are meant to represent sets, as we make precise by defining \emph{interpretations} $\llbracket\cdot\rrbracket$ as certain maps on terms determined by some assignment of sets to the variables and constants.  We then define \emph{statements} $\mathsf{S}$ formed from terms.  These represent actual mathematical statements about our sets which either hold or fail in a particular interpretation.  This in turn is made precise by showing how interpretations $\llbracket\cdot\rrbracket$ determine a subset of statements $\Sigma$ \emph{satisfied} by $\llbracket\cdot\rrbracket$.  In this case we call $\llbracket\cdot\rrbracket$ a \emph{model} for $\Sigma$.

This naturally leads to a semantic \emph{consequence relation} ${\vDash}\subseteq\mathcal{P}(\mathsf{S})\times\mathsf{S}$, where $\Gamma\vDash X$ indicates that every model for all the statements in $\Gamma$ is a model for the statement $X$.  We then investigate syntactic properties of $\vDash$ which in the end we take as inference rules defining a syntactic \emph{inference relation} ${\vdash}\subseteq\mathcal{P}(\mathsf{S})\times\mathsf{S}$, i.e.~$\vdash$ is the smallest relation obeying all the inference rules.  As we derive the inference rules from properties of $\vDash$ from the outset, the inference relation will automatically be \emph{sound} with respect to the consequence relation, i.e.~${\vdash}\subseteq{\vDash}$.  Finally we examine properties of $\vdash$ and discuss the axioms that would be needed to use $\vdash$ as a foundation for mathematics.

\part{A Bare Bones Dependent Type System}\label{Part1}

\section{Terms}

The alphabet of our first system has 6 symbols which we denote by
\[\mathsf{A}=\{\mathsf{c},\mathsf{v},{'},\rho,\beta,\lambda\}.\]
The strings we are interested in are defined in Backus-Naur form by
\begin{align*}
    \tag{Constants}\mathsf{C}&::= \mathsf{c}\ |\ \mathsf{C}'\\
    \tag{Variables}\mathsf{V}&::= \mathsf{v}\ |\ \mathsf{V}'\\
    \tag{Terms}\mathsf{T}&::= \mathsf{C}\ |\ \mathsf{V}\ |\ \rho\mathsf{T}\ |\ \beta\mathsf{TT}\ |\ \lambda\mathsf{VTT}
\end{align*}
More explicitly, the set of constant and variable strings are given by
\begin{align*}
    \mathsf{C}&=\{\mathsf{c},\mathsf{c}',\mathsf{c}'',\ldots\}\quad\text{and}\\
    \mathsf{V}&=\{\mathsf{v},\mathsf{v}',\mathsf{v}'',\ldots\},
\end{align*} 
while the terms are the smallest set $\mathsf{T}$ of strings such that
\[\mathsf{C}\cup\mathsf{V}\cup\rho\mathsf{T}\cup\beta\mathsf{TT}\cup\lambda\mathsf{VTT}\subseteq\mathsf{T}.\]

The $\rho$, $\beta$ and $\lambda$ terms here are meant to represent dependent products, applications and abstractions respectively, as will soon be clear from the interpretations we define below.  Terms like these are standard in dependent type systems, although our formal syntax differs a little from that commonly seen in the literature.  For example, our $\beta$-terms would usually just be written by juxtapostion without the $\beta$, i.e.~$\beta RS$ would just be $RS$.  We are adding the $\beta$ so our formal language requires no parentheses, a side-benefit being that $\beta$-reduction then does indeed amount to reducing $\beta$-terms.  Also $\lambda$-terms are often written with extra colons, dots and arrows, but for simplicity's sake we prefer to avoid unnecessary punctuation, at least in the formal language.  Likewise, products are usually written like $\lambda$-terms just with their own binder $\pi$ replacing $\lambda$.  We prefer a simpler language with a single binder $\lambda$, but we can still view $\pi$ as an abbreviation (so-called `syntactic sugar') for $\rho\lambda$.  We will discuss more abbreviations below.

Before moving on, let us stratify the terms based on the number of operations needed to form them.  Specifically, let $\mathsf{T}^0=\mathsf{C}\cup\mathsf{V}$ and
\[\mathsf{T}^{n+1}=\mathsf{T}^0\cup\rho\mathsf{T}^n\cup\beta\mathsf{T}^n\mathsf{T}^n\cup\lambda\mathsf{V}\mathsf{T}^n\mathsf{T}^n,\]
for all $n\in\omega$.  So $\mathsf{T}^0$ consists of the \emph{atomic} terms and $\mathsf T=\bigcup_{n\in\omega}\mathsf T^n$.  As we are taking strings in $\mathsf{V}$ and $\mathsf{C}$ as atomic, we do not want to view $\mathsf{v}'$ as a substring of $\mathsf{v}''$, for example.  Accordingly we define the substring relation $\sqsubseteq$ on $\mathsf{T}$ as follows.  First, for all $Q,R,S,T\in\mathsf{A}^{<\omega}$, we define
\[S\sqsubseteq_T^QR\qquad\Leftrightarrow\qquad R=QST\text{ and }T(0)\neq{'}.\]
Then we define
\[S\sqsubseteq R\qquad\Leftrightarrow\qquad S\sqsubseteq_T^QR,\text{ for some }Q,T\in\mathsf{A}^{<\omega}.\]

\section{Interpretations}

Now that we have defined our terms, a type theorist might expect us to immediately define statements and list the inference rules that can be applied to them.  However, we would first like to make our intended interpretations of the terms clear.  Only then does it make sense to introduce statements and inference rules based on these interpretations.

To interpret terms, we first assign sets to variables and constants.  Accordingly, we call a function $\llbracket\cdot\rrbracket^0$ on $\mathsf{T}^0$ an \emph{assignment}.  Given an assignment $\llbracket\cdot\rrbracket^0$, we can modify it to agree with a function $\psi$ defined on a subset of $\mathsf{T}^0$.  We denote this modification by $\llbracket\cdot\rrbracket^0_\psi$, i.e.
\[\llbracket t\rrbracket^0_\psi=\begin{cases}\psi(t)&\text{if }t\in\mathrm{dom}(f)\\\llbracket t\rrbracket^0&\text{otherwise}.\end{cases}\]
In particular, given any $x\in\mathsf{V}$ and $s$, we get another assignment $\llbracket\cdot\rrbracket^0_{\{\langle s,x\rangle\}}$ which we abbreviate to $\llbracket\cdot\rrbracket^0_{\langle s,x\rangle}$ so that
\[\llbracket t\rrbracket^0_{\langle s,x\rangle}=\begin{cases}s&\text{if }t=x\\\llbracket t\rrbracket^0&\text{if }t\in\mathsf{T}^0\setminus\{x\}.\end{cases}\]

\begin{definition}
    We extend any assignment $\llbracket\cdot\rrbracket^0$ on $\mathsf{T}^0$ to functions $\llbracket\cdot\rrbracket^n=(\llbracket\cdot\rrbracket^0)^n$ on $\mathsf{T}^n$, for each $n\in\omega$, so that
    \begin{align*}
        \llbracket\rho S\rrbracket^{n+1}&=\prod\llbracket S\rrbracket^n,\\
        \llbracket\beta RS\rrbracket^{n+1}&=\llbracket R\rrbracket^n(\llbracket S\rrbracket^n)\quad\text{and}\\
        \llbracket\lambda xRS\rrbracket^{n+1}&=\{\langle\llbracket S\rrbracket_{\langle r,x\rangle}^n,r\rangle\mathrel{|}r\in\llbracket R\rrbracket^n\},
    \end{align*}
    for $R,S\in\mathsf{T}^n$ and $x\in\mathsf{V}$, where $\llbracket\cdot\rrbracket_{\langle r,x\rangle}^n=(\llbracket\cdot\rrbracket^0_{\langle r,x\rangle})^n$.  The \emph{interpretation} coming from $\llbracket\cdot\rrbracket^0$ is the function $\llbracket\cdot\rrbracket=\llbracket\cdot\rrbracket^\omega=(\llbracket\cdot\rrbracket^0)^\omega$ on $\mathsf{T}$ defined by
    \[\llbracket\cdot\rrbracket=\bigcup_{n\in\omega}\llbracket\cdot\rrbracket^n.\]
\end{definition}

Restricting an interpretation $\llbracket\cdot\rrbracket$ to $\mathsf{T}^0$ is immediately seen to recover the underlying assignment $\llbracket\cdot\rrbracket^0$, from which we then define
\[\llbracket\cdot\rrbracket_\psi=\bigcup_{n\in\omega}\llbracket\cdot\rrbracket^n_\psi=\bigcup_{n\in\omega}(\llbracket S\rrbracket^0_\psi)^n,\]
We then see that the above defining properties of $\llbracket\cdot\rrbracket^n$ on $\mathsf{T}^n$ remain valid for the whole interpretation $\llbracket\cdot\rrbracket$ on $\mathsf{T}$, i.e.
    \begin{align*}
        \llbracket\rho S\rrbracket&=\prod\llbracket S\rrbracket,\\
        \llbracket\beta RS\rrbracket&=\llbracket R\rrbracket(\llbracket S\rrbracket)\quad\text{and}\\
        \llbracket\lambda xRS\rrbracket&=\{\langle\llbracket S\rrbracket_{\langle r,x\rangle},r\rangle\mathrel{|}r\in\llbracket R\rrbracket\}.
    \end{align*}
Our interpretations are thus consistent with those defined in \cite{Aczel1999}.  This contrasts with other notions of interpretation, like those given in \cite{MiquelWerner2003}, which further depend on a given set of assumptions $\Gamma$.

Note that the above means that a $\lambda$-term $\lambda xRS$ is interpreted as a function with domain $\llbracket R\rrbracket$ such that, for all $r\in\llbracket R\rrbracket$,
\[\llbracket\lambda xRS\rrbracket(r)=\llbracket S\rrbracket_{\langle r,x\rangle}.\]
In particular, if we have another term $T\in\mathsf{T}$ with $\llbracket T\rrbracket\in\llbracket R\rrbracket$ then
\[\llbracket\beta\lambda xRST\rrbracket=\llbracket\lambda xRS\rrbracket(\llbracket T\rrbracket)=\llbracket S\rrbracket_{\langle\llbracket T\rrbracket,x\rangle}.\]
Based on our definition of interpretations, $\llbracket S\rrbracket_{\langle\llbracket T\rrbracket,x\rangle}$ should be the same as the interpretation of another term $S_{[T/x]}$ obtained by substituting $T$ for all `free' occurrences of $x$ in $S$, as long as we are careful to rename bound variables in $S$ so the binder $\lambda$ does not unintentionally capture any free variables in $T$.  Let us now make this more precise.

\section{Free Variables}

Intuitively, a variable $x$ should be free in a term $T$ if changing an interpretation at $x$ could conceivably change the interpretation of $T$.  Accordingly, we define the \emph{free variables} $\mathsf{F}(T)\subseteq\mathsf{V}$ of any term $T$ so that, for all $a\in\mathsf{C}$, $x\in\mathsf{V}$ and $R,S\in\mathsf{T}$,
\begin{align*}
    \mathsf{F}(a)&=\emptyset.\\
    \mathsf{F}(x)&=\{x\}.\\
    \mathsf{F}(\rho R)&=\mathsf{F}(R).\\
    \mathsf{F}(\beta RS)&=\mathsf{F}(R)\cup\mathsf{F}(S).\\
    \mathsf{F}(\lambda xRS)&=\mathsf{F}(R)\cup(\mathsf{F}(S)\setminus\{x\}).
\end{align*}

\begin{proposition}\label{FreeProp}
    For any interpretation $\llbracket\cdot\rrbracket$, $S\in\mathsf{T}$, $x\in\mathsf{V}$ and $s$,
    \[x\notin\mathsf{F}(S)\qquad\Rightarrow\qquad\llbracket S\rrbracket=\llbracket S\rrbracket_{\langle s,x\rangle}.\]
\end{proposition}

\begin{proof}
    First note that if $t\in\mathsf{T}^0$ and $x\in\mathsf{V}\setminus\mathsf{F}(t)$ then $x\neq t$ and hence $\llbracket t\rrbracket_{\langle s,x\rangle}=\llbracket t\rrbracket$.  Now assume the result holds for any $R,S\in\mathsf{T}^n$.  If $x\in\mathsf{V}\setminus\mathsf{F}(R)$ then $\llbracket\rho R\rrbracket_{\langle s,x\rangle}=\prod\llbracket R\rrbracket_{\langle s,x\rangle}=\prod\llbracket R\rrbracket=\llbracket\rho R\rrbracket$.  Likewise, if $x\in\mathsf{V}\setminus(\mathsf{F}(R)\cup\mathsf{F}(S))$ then $\llbracket\beta RS\rrbracket_{\langle s,x\rangle}=\llbracket R\rrbracket_{\langle s,x\rangle}(\llbracket S\rrbracket_{\langle s,x\rangle})=\llbracket\beta RS\rrbracket$.  Now let $\llbracket\cdot\rrbracket_{\langle r,x\rangle\langle s,y\rangle}=(\llbracket\cdot\rrbracket_{\langle r,x\rangle})_{\langle s,y\rangle}$ denote $\llbracket\cdot\rrbracket$ changed first at $x$ to $r$ and then at $y$ to $s$.  If $x\neq y$, we can reverse the order, i.e.
    \[x\neq y\qquad\Rightarrow\qquad\llbracket\cdot\rrbracket_{\langle r,x\rangle\langle s,y\rangle}=\llbracket\cdot\rrbracket_{\langle s,y\rangle\langle r,x\rangle}.\]
    But if $x=y$ then the second change cancels the first, i.e.
    \[\llbracket\cdot\rrbracket_{\langle r,x\rangle\langle s,x\rangle}=\llbracket\cdot\rrbracket_{\langle s,x\rangle}.\]
    Thus if $x\in\mathsf{V}\setminus(\mathsf{F}(R)\cup(\mathsf{F}(S)\setminus\{y\}))$ then
    \begin{align*}
        \llbracket\lambda yRS\rrbracket_{\langle s,x\rangle}&=\{\langle\llbracket S\rrbracket_{\langle s,x\rangle\langle r,y\rangle},r\rangle:r\in\llbracket R\rrbracket_{\langle s,x\rangle}\}\\
        &=\{\langle\llbracket S\rrbracket_{\langle r,y\rangle},r\rangle:r\in\llbracket R\rrbracket\}.\\
        &=\llbracket\lambda yRS\rrbracket.
    \end{align*}
    Indeed, if $x=y$ then $\llbracket S\rrbracket_{\langle s,x\rangle\langle r,y\rangle}=\llbracket S\rrbracket_{\langle s,y\rangle\langle r,y\rangle}=\llbracket S\rrbracket_{\langle r,y\rangle}$, while if $x\neq y$ then $x\notin\mathsf{F}(S)$ so $\llbracket S\rrbracket_{\langle s,x\rangle\langle r,y\rangle}=\llbracket S\rrbracket_{\langle r,y\rangle\langle s,x\rangle}=\llbracket S\rrbracket_{\langle r,y\rangle}$, by the inductive hypothesis.  This proves the result for all terms in $\mathsf{T}^{n+1}$, as required.
\end{proof}

Likewise, we define the \emph{bound variables} $\mathsf{B}(T)$ of any term $T$ so that, for all $t\in\mathsf{T}^0$, $x\in\mathsf{V}$ and $R,S\in\mathsf{T}$,
\begin{align*}
    \mathsf{B}(t)&=\emptyset.\\
    \mathsf{B}(\rho R)&=\mathsf{B}(R).\\
    \mathsf{B}(\beta RS)&=\mathsf{B}(R)\cup\mathsf{B}(S).\\
    \mathsf{B}(\lambda xRS)&=\mathsf{B}(R)\cup\mathsf{B}(S)\cup\{x\}.
\end{align*}
Equivalently, $x\in\mathsf{B}(T)$ if it appears next to the binder $\lambda$ in $T$, i.e.
\[\mathsf{B}(T)=\{x\in\mathsf{V}\mathrel{|}\lambda x\sqsubseteq T\}.\]
The set of all variables appearing in $T$ will be denoted by
\[\mathsf{V}(T)=\mathsf{F}(T)\cup\mathsf{B}(T)=\{x\in\mathsf{V}\mathrel{|}x\sqsubseteq T\}.\]

\section{Substitution}

To define substitution on $\mathsf{T}^0$, take any $x\in\mathsf{V}$ and $T\in\mathsf{T}$ and let
\[t_{[T/x]}:=\begin{cases}T&\text{if }t=x\\t&\text{if }t\in\mathsf{T}^0\setminus\{x\}.\end{cases}\]
Then extend from $\mathsf{T}^n$ to $\mathsf{T}^{n+1}$ by defining, for all $R,S\in\mathsf{T}^n$ and $y\in\mathsf{V}$,
\begin{align*}
    (\rho R)_{[T/x]}&:=\rho R_{[T/x]},\\
    (\beta RS)_{[T/x]}&:=\beta R_{[T/x]}S_{[T/x]}\quad\text{and}\\
    (\lambda yRS)_{[T/x]}&:=\begin{cases}\lambda yR_{[T/x]}S&\text{if }y=x\text{ or }x\notin\mathsf{F}(S)\\\lambda zR_{[T/x]}S_{[z/y][T/x]}&\text{otherwise, where }z\notin\mathsf{F}(\lambda yTS).\end{cases}
\end{align*}
As far as interpretations are concerned, it does not matter which variable outside of $\mathsf{F}(\lambda yTS)=\mathsf{F}(T)\cup(\mathsf{F}(S)\setminus\{y\})$ we take for $z$ above.  However, for some basic syntactic properties, it will be important to avoid unnecessary changes of bound variables.  Accordingly, let us take $z=y$ above whenever $y\notin\mathsf{F}(T)$ (otherwise take $z$ to be the shortest variable outside of $\mathsf{V}(TS)$, for example).  This ensures that
\[x\neq y\notin\mathsf{F}(T)\quad\Rightarrow\quad(\lambda yRS)_{[T/x]}=\lambda y R_{[T/x]}S_{[T/x]}.\]
This completes the recursive definition of substitution on $\mathsf{T}$.

Our original motivation for defining free variables was in terms of changing interpretations.  However, freeness is more accurately characterised by changing under substitution, specifically
\[\mathsf{F}(S)=\{x\in\mathsf{V}\mathrel{|}\forall T\in\mathsf{T}\setminus\{x\}\ (S_{[T/x]}\neq S)\}.\]
This is immediate from the following result.

\begin{proposition}\label{FreeChar}
        For all $x\in\mathsf{V}$ and $S,T\in\mathsf{T}$,
        \[T\neq x\in\mathsf{F}(S)\qquad\Leftrightarrow\qquad S_{[T/x]}\neq S.\]
\end{proposition}

\begin{proof}
    First we claim that
    \[x\notin\mathsf{F}(S)\quad\Rightarrow\quad S_{[T/x]}=S.\]
    To see this, note that if $t\in\mathsf{T}^0$ then $x\notin\mathsf{F}(t)\subseteq\{t\}$ means $t_{[T/x]}=t$.  Now assume the claim holds for all $S\in\mathsf{T}^n$ and take any $R,S\in\mathsf{T}^n$.  If $x\notin\mathsf{F}(\rho R)=\mathsf{F}(R)$ then $(\rho R)_{[T/x]}=\rho R_{[T/x]}=\rho R$.  Likewise, if $x\notin\mathsf{F}(\beta RS)=\mathsf{F}(S)\cup\mathsf{F}(R)$ then $(\beta RS)_{[T/x]}=\beta R_{[T/x]}S_{[T/x]}=\beta RS$.  If $x\notin\mathsf{F}(\lambda yRS)$ then either $y=x\notin\mathsf{F}(R)$ or $y\neq x\notin\mathsf{F}(R)\cup\mathsf{F}(S)$ and so in both cases $(\lambda yRS)_{[T/x]}=\lambda yR_{[T/x]}S=\lambda yRS$.  This proves the claim on $\mathsf{T}^{n+1}$ and hence on $\mathsf{T}$, by induction.
    
    An even simpler inductive argument proves $S_{[x/x]}=S$, for all $S\in\mathsf{T}$, the only slightly non-trivial thing to note is that $y\neq x$ implies $y\notin\mathsf{F}(x)$ and hence $(\lambda yRS)_{[x/x]}=\lambda yR_{[x/x]}S_{[x/x]}=\lambda yRS$.  This completes the proof of the $\Leftarrow$ part of the result.

    For the $\Rightarrow$ part first note $t\in\mathsf{T}^0$ and $T\neq x\in\mathsf{F}(t)=\{t\}$ implies $t_{[T/x]}=x_{[T/x]}=T\neq t$, proving $\Rightarrow$ on $\mathsf{T}^0$.  Now assume $\Rightarrow$ holds for all $S\in\mathsf{T}^n$ and take any $R,S\in\mathsf{T}^n$.  Then $T\neq x\in\mathsf{F}(\rho R)=\mathsf{F}(R)$ implies $(\rho R)_{[T/x]}=\rho R_{[T/x]}\neq\rho R$.  Also $T\neq x\in\mathsf{F}(\beta RS)=\mathsf{F}(R)\cup\mathsf{F}(S)$ implies $R\neq R_{[T/x]}$ or $S\neq S_{[T/x]}$ and hence $(\beta RS)_{[T/x]}=\beta R_{[T/x]}S_{[T/x]}\neq\beta RS$.  Now if $T\neq x\in\mathsf{F}(\lambda yRS)$ then $y=x\in\mathsf{F}(R)$ or $y\neq x\in\mathsf{F}(R)\cup\mathsf{F}(S)$.  In the former case $(\lambda yRS)_{[T/x]}=\lambda yR_{[T/x]}S\neq\lambda yRS$ while in the latter case either $R_{[T/x]}\neq R$ or $S_{[T/x]}\neq S$ and hence $S_{[z/y][T/x]}\neq S$, whenever $z\notin\mathsf{F}(\lambda yTS)$, which yields $(\lambda yRS)_{[T/x]}=\lambda zR_{[T/x]}S_{[z/y][T/x]}\neq\lambda yRS$.  This proves $\Rightarrow$ on $\mathsf{T}^{n+1}$ and hence on $\mathsf{T}$, by induction.    
\end{proof}

We can now verify that substitutions have the desired interpretation.

\begin{proposition}\label{SubstitutivityProp}
    For any interpretation $\llbracket\cdot\rrbracket$, $x\in\mathsf{V}$ and $S,T\in\mathsf{T}$,
    \begin{equation}\label{Substitutivity}
        \tag{Substitutivity}\llbracket S_{[T/x]}\rrbracket=\llbracket S\rrbracket_{\langle\llbracket T\rrbracket,x\rangle}.
    \end{equation}
\end{proposition}

\begin{proof}
    Certainly $\llbracket x\rrbracket_{\langle\llbracket T\rrbracket,x\rangle}=\llbracket T\rrbracket=\llbracket x_{[T/x]}\rrbracket$ and, for any $t\in\mathsf{T}^0\setminus\{x\}$, $\llbracket t\rrbracket_{\langle\llbracket T\rrbracket,x\rangle}=\llbracket t\rrbracket=\llbracket t_{[T/x]}\rrbracket$, proving the result for $S\in\mathsf{T}^0$.
    
    Now assume the result for all $S\in\mathsf{T}^n$ and take any $R,S\in\mathsf{T}^n$.  Note
    \[\llbracket\rho R\rrbracket_{\langle\llbracket T\rrbracket,x\rangle}=\prod\llbracket R\rrbracket_{\langle\llbracket T\rrbracket,x\rangle}=\prod\llbracket R_{[T/x]}\rrbracket=\llbracket\rho R_{[T/x]}\rrbracket=\llbracket(\rho R)_{[T/x]}\rrbracket.\]
    Likewise, $\llbracket\beta RS\rrbracket_{\langle\llbracket T\rrbracket,x\rangle}=\llbracket\beta R_{[T/x]}S_{[T/x]}\rrbracket=\llbracket(\beta RS)_{[T/x]}\rrbracket$.  Next note    
    \begin{align*}
        \llbracket\lambda xRS\rrbracket_{\langle\llbracket T\rrbracket,x\rangle}&=\{\langle\llbracket S\rrbracket_{\langle\llbracket T\rrbracket,x\rangle\langle r,x\rangle},r\rangle\mathrel{|}r\in\llbracket R\rrbracket_{\langle\llbracket T\rrbracket,x\rangle}\}\\
        &=\{\langle\llbracket S\rrbracket_{\langle r,x\rangle},r\rangle\mathrel{|}r\in\llbracket R_{[T/x]}\rrbracket\}\\
        &=\llbracket\lambda xR_{[T/x]}S\rrbracket\\
        &=\llbracket(\lambda xRS)_{[T/x]}\rrbracket.
    \end{align*}
    Now take any $y\in\mathsf{V}\setminus\{x\}$.  If $x\notin\mathsf{F}(S)$ then \Cref{FreeProp} yields $\llbracket S\rrbracket_{\langle\llbracket T\rrbracket,x\rangle\langle r,y\rangle}=\llbracket S\rrbracket_{\langle r,y\rangle\langle\llbracket T\rrbracket,x\rangle}=\llbracket S\rrbracket_{\langle r,y\rangle}$,  and so again
    \begin{align*}
        \llbracket\lambda yRS\rrbracket_{\langle\llbracket T\rrbracket,x\rangle}&=\{\langle\llbracket S\rrbracket_{\langle\llbracket T\rrbracket,x\rangle\langle r,y\rangle},r\rangle\mathrel{|}r\in\llbracket R\rrbracket_{\langle\llbracket T\rrbracket,x\rangle}\}\\
        &=\{\langle\llbracket S\rrbracket_{\langle r,y\rangle},r\rangle\mathrel{|}r\in\llbracket R_{[T/x]}\rrbracket\}\\
        &=\llbracket\lambda yR_{[T/x]}S\rrbracket\\
        &=\llbracket(\lambda yRS)_{[T/x]}\rrbracket.
    \end{align*}
    Otherwise $y\neq x\in\mathsf{F}(S)$.  For any $z\notin\mathsf{F}(\lambda yTS)(\ni x)$ and $r$, then $\llbracket T\rrbracket=\llbracket T\rrbracket_{\langle r,z\rangle}$.  Also, for any interpretation $\llparenthesis\cdot\rrparenthesis$, $\llparenthesis S_{[z/y]}\rrparenthesis_{\langle r,z\rangle}=\llparenthesis S\rrparenthesis_{\langle r,y\rangle}$ -- If $z=y$ then this is immediate, otherwise $y\neq z\notin\mathsf{F}(S)$ and hence
    \[\llparenthesis S_{[z/y]}\rrparenthesis_{\langle r,z\rangle}=\llparenthesis S\rrparenthesis_{\langle r,z\rangle\langle\llparenthesis z\rrparenthesis_{\langle r,z\rangle},y\rangle}=\llparenthesis S\rrparenthesis_{\langle r,z\rangle\langle r,y\rangle}=\llparenthesis S\rrparenthesis_{\langle r,y\rangle\langle r,z\rangle}=\llparenthesis S\rrparenthesis_{\langle r,y\rangle}.\]
    In particular we can take $\llparenthesis\cdot\rrparenthesis=\llbracket\cdot\rrbracket_{\langle\llbracket T\rrbracket,x\rangle}=\llbracket\cdot\rrbracket_{\langle\llbracket T\rrbracket_{\langle r,z\rangle},x\rangle}$.  Noting $x\neq z$,
    \begin{align*}
        \llbracket S\rrbracket_{\langle\llbracket T\rrbracket,x\rangle\langle r,y\rangle}&=\llbracket S_{[z/y]}\rrbracket_{\langle\llbracket T\rrbracket_{\langle r,z\rangle},x\rangle\langle r,z\rangle}=\llbracket S_{[z/y]}\rrbracket_{\langle r,z\rangle\langle\llbracket T\rrbracket_{\langle r,z\rangle},x\rangle}\\
        &=\llbracket S_{[z/y][T/x]}\rrbracket_{\langle r,z\rangle}.
    \end{align*}
    From this it follows that
    \begin{align*}
        \llbracket\lambda yRS\rrbracket_{\langle\llbracket T\rrbracket,x\rangle}&=\{\langle\llbracket S\rrbracket_{\langle\llbracket T\rrbracket,x\rangle\langle r,y\rangle},r\rangle\mathrel{|}r\in\llbracket R\rrbracket_{\langle\llbracket T\rrbracket,x\rangle}\}\\
        &=\{\langle\llbracket S_{[z/y][T/x]}\rrbracket_{\langle r,z\rangle},r\rangle\mathrel{|}r\in\llbracket R_{[T/x]}\rrbracket\}\\
        &=\llbracket\lambda zR_{[T/x]}S_{[z/y][T/x]}\rrbracket=\llbracket(\lambda yRS)_{[T/x]}\rrbracket.
    \end{align*}
    This completes the induction.
\end{proof}

Here is another observation on substitutions that is sometimes useful.

\begin{proposition}\label{FreeSub}
    For all $x\in\mathsf{V}$ and $S,T\in\mathsf{T}$,
    \[x\in\mathsf{F}(S)\quad\Rightarrow\quad\mathsf{F}(S_{[T/x]})=\mathsf{F}(\lambda xTS).\]
\end{proposition}

\begin{proof}
    If $t\in\mathsf{T}^0$ then $x\in\mathsf{F}(t)\subseteq\{t\}$ means $t_{[T/x]}=x_{[T/x]}=T$ so $\mathsf{F}(t_{[T/x]})=\mathsf{F}(T)=\mathsf{F}(\lambda xTt)$.  This proves the result on $\mathsf{T}^0$.  Now assume the result holds on $\mathsf{T}^n$ and take $R,S\in\mathsf{T}^n$.  If $x\in\mathsf{F}(\rho R)=\mathsf{F}(R)$ then
    \[\mathsf{F}((\rho R)_{[T/x]})=\mathsf{F}(\rho R_{[T/x]})=\mathsf{F}(R_{[T/x]})=\mathsf{F}(\lambda xTR)=\mathsf{F}(\lambda xT\rho R).\]
    Likewise, if $x\in\mathsf{F}(\beta RS)=\mathsf{F}(R)\cup\mathsf{F}(S)$ then either $x\in\mathsf{F}(S)$ or $x\in\mathsf{F}(R)$ but in either case $\mathsf{F}((\beta RS)_{[T/x]})=\mathsf{F}(\beta R_{[T/x]}S_{[T/x]})$ contains $\mathsf{F}(T)$ so
    \[\mathsf{F}((\beta RS)_{[T/x]})=(\mathsf{F}(R)\setminus\{x\})\cup(\mathsf{F}(S)\setminus\{x\})\cup\mathsf{F}(T)=\mathsf{F}(\lambda xT\beta RS).\]
    Now if $x\in\mathsf{F}(\lambda yRS)$ then either $y=x\in\mathsf{F}(R)$ or $y\neq x\in\mathsf{F}(R)\cup\mathsf{F}(S)$.  In the former case
    \begin{align*}
        \mathsf{F}((\lambda yRS)_{[T/x]})&=\mathsf{F}(\lambda yR_{[T/x]}S)\\
        &=\mathsf{F}(T)\cup(\mathsf{F}(R)\setminus\{x\})\cup(\mathsf{F}(S)\setminus\{y\})\\
        &=\mathsf{F}(T)\cup(\mathsf{F}(\lambda yRS)\setminus\{x\}),
    \end{align*}
    as required.  Similarly, in the latter case
    \begin{align*}
        \mathsf{F}((\lambda yRS)_{[T/x]})&=\mathsf{F}(\lambda zR_{[T/x]}S_{[z/y][T/x]})\\
        &=\mathsf{F}(T)\cup(\mathsf{F}(R)\setminus\{x\})\cup(\mathsf{F}(S)\setminus\{x,y\})\\
        &=\mathsf{F}(T)\cup(\mathsf{F}(\lambda yRS)\setminus\{x\}).
    \end{align*}
    This proves the result on $\mathsf{T}^{n+1}$ and hence on $\mathsf{T}$, by induction.
\end{proof}

\section{Conversion}

We can now show that changing bound variables does not affect the interpretation of the resulting term.

\begin{proposition}\label{AlphaEqual}
    For any interpretation $\llbracket\cdot\rrbracket$, $x,y\in\mathsf{V}$ and $R,S\in\mathsf{T}$,
    \[y\notin\mathsf{F}(S)\quad\Rightarrow\quad\llbracket\lambda xRS\rrbracket=\llbracket\lambda yRS_{[y/x]}\rrbracket.\]
\end{proposition}

\begin{proof}
    First note \eqref{Substitutivity} yields
    \[\llbracket S_{[y/x]}\rrbracket_{\langle r,y\rangle}=\llbracket S\rrbracket_{\langle r,y\rangle\langle\llbracket y\rrbracket_{\langle r,y\rangle},x\rangle}=\llbracket S\rrbracket_{\langle r,y\rangle\langle r,x\rangle}=\llbracket S\rrbracket_{\langle r,x\rangle},\]
    by \Cref{FreeProp}, as $y\notin\mathsf{F}(S)$.  Thus
    \begin{align*}
        \llbracket\lambda xRS\rrbracket&=\{\langle\llbracket S\rrbracket_{\langle r,x\rangle},r\rangle\mathrel{|}r\in\llbracket R\rrbracket\}\\
        &=\{\langle\llbracket S_{[y/x]}\rrbracket_{\langle r,y\rangle},r\rangle\mathrel{|}r\in\llbracket R\rrbracket\}\\
        &=\llbracket\lambda yRS_{[y/x]}\rrbracket.\qedhere
    \end{align*}
\end{proof}

If we even have $y\notin\mathsf{V}(S)$ then we can reverse this process.

\begin{proposition}\label{AlphaReverse}
    For any $x,y\in\mathsf{V}$ and $S\in\mathsf{T}$,
    \[y\notin\mathsf{V}(S)\quad\Rightarrow\quad S=S_{[y/x][x/y]}.\]
\end{proposition}

\begin{proof}
    Note $x_{[y/x][x/y]}=y_{[x/y]}=x$.  If $t\in\mathsf{T}^0\setminus\{x\}$ and $y\notin\mathsf{V}(t)$ and hence $y\neq t\neq x$ then $t_{[y/x][x/y]}=t_{[x/y]}=t$, proving the result for $\mathsf{T}^0$.  Now assume the result for $\mathsf{T}^n$ and take $R,S\in\mathsf{T}$.  Then $y\notin\mathsf{V}(\lambda xRS)$ implies $x\neq y\notin\mathsf{F}(S)$ and hence
    \[(\lambda xRS)_{[y/x][x/y]}=(\lambda xR_{[y/x]}S)_{[x/y]}=\lambda xR_{[y/x][x/y]}S_{[x/y]}=\lambda xRS.\]
    Also, for any $z\in\mathsf{V}\setminus\{x\}$, $y\notin\mathsf{V}(\lambda zRS)$ implies $y\neq z$ and hence
    \[(\lambda zRS)_{[y/x][x/y]}=(\lambda zR_{[y/x]}S_{[y/x]})_{[x/y]}=\lambda zR_{[y/x][x/y]}S_{[y/x][x/y]}=\lambda xRS.\]
    This completes the proof for $\lambda$-terms in $\mathsf{T}^{n+1}$, while for $\rho$ and $\beta$ terms this is immediate.  Now the result on $\mathsf{T}$ holds by induction.
\end{proof}

This motivates the definition of ${\equiv^\alpha}\subseteq\mathsf{T}\times\mathsf{T}$ by
\[{\equiv^\alpha}=\{\langle\lambda xRS,\lambda yRS_{[y/x]}\rangle:x\notin\mathsf{B}(S)\text{ and }y\notin\mathsf{V}(S)\}.\]
We call this $\equiv^\alpha$ the \emph{head $\alpha$-conversion} relation.

\begin{corollary}
    The head $\alpha$-conversion relation $\equiv^\alpha$ is symmetric.
\end{corollary}

\begin{proof}
    If $R,S\in\mathsf{T}$, $x\in\mathsf{V}$ and $y\in\mathsf{V}\setminus\mathsf{V}(S)$, \Cref{AlphaReverse} yields
    \[\langle\lambda yRS_{[y/x]},\lambda xRS\rangle\ =\ \langle\lambda yRS_{[y/x]},\lambda xRS_{[y/x][x/y]}\rangle\ \in\ {\equiv^\alpha}.\]
    Indeed $\mathsf{B}(S)=\mathsf{B}(S_{[y/x]})$, as $y\notin\mathsf{V}(S)$, so $y\notin\mathsf{V}(S)\supseteq\mathsf{B}(S)=\mathsf{B}(S_{[y/x]})$, $x\notin\mathsf{B}(S)=\mathsf{B}(S_{[y/x]})$ and $x\notin\mathsf{F}(S_{[y/x]})$, by \Cref{FreeSub}.
\end{proof}

Given ${\vartriangleright}\subseteq\mathsf{T}\times\mathsf{T}$, define its \emph{contextual closure} ${\vartriangleright^\mathsf{c}}\subseteq\mathsf{T}\times\mathsf{T}$ by
\[R'\vartriangleright^\mathsf{c}C'\quad\Leftrightarrow\quad\exists P,Q,R,C\ (R\sqsubseteq_Q^PR',C\sqsubseteq_Q^PC'\text{ and }R\vartriangleright C).\]
Note $P$ and $Q$ above are not assumed to be terms, just strings (possibly even empty strings) which become terms when added to either side of $R$ and $C$.  Equivalently, $\vartriangleright^\mathsf{c}$ is the smallest relation on $\mathsf{T}$ containing $\vartriangleright$ such that, for all $x\in\mathsf{V}$ and $R,S,C\in\mathsf{T}$, $R\vartriangleright^\mathsf{c}C$ always implies
\[\beta SR\vartriangleright^\mathsf{c}\beta SC,\beta RS\vartriangleright^\mathsf{c}\beta CS,\lambda x SR\vartriangleright^\mathsf{c}\lambda xSC,\lambda xRS\vartriangleright^\mathsf{c}\lambda xCS\text{ and }\rho R\vartriangleright^\mathsf{c}\rho C.\]

We also define the reflexive closure of any ${\vartriangleright^\mathsf{r}}\subseteq\mathsf{T}\times\mathsf{T}$ by
\[R\vartriangleright^\mathsf{r}C\quad\Leftrightarrow\quad R\vartriangleright C\text{ or }R=C.\]
Note that if $\llbracket R\rrbracket=\llbracket C\rrbracket$ whenever $R\vartriangleleft C$ then it also follows that $\llbracket R\rrbracket=\llbracket C\rrbracket$ whenever $R\vartriangleleft^\mathsf{ctr}C$.  Defining \emph{$\alpha$-conversion} $\equiv_\alpha$ as the reflexive transitive contextual closure of head $\alpha$-conversion
\[\tag{$\alpha$-Conversion}{\equiv_\alpha}\ =\ {\equiv^{\alpha\mathsf{ctr}}},\]
the following is thus immediate from \Cref{AlphaEqual}.

\begin{corollary}\label{AlphaEqual2}
    For any interpretation $\llbracket\cdot\rrbracket$ and $R,C\in\mathsf{T}$,
    \[R\equiv_\alpha C\qquad\Rightarrow\qquad\llbracket R\rrbracket=\llbracket C\rrbracket.\]
\end{corollary}

As $\equiv^\alpha$ is symmetric, so is $\equiv_\alpha$ and hence we could alternatively describe $\equiv_\alpha$ as the equivalence relation on $\mathsf{T}$ generated by $\equiv^{\alpha\mathsf{c}}$.  Let us denote the corresponding equivalence class of any $S\in\mathsf{T}$ by
\[[S]_\alpha\ =\ {\equiv_\alpha}[S]\ =\ \{R\in\mathsf{T}\mathrel{|}R\equiv_\alpha S\}.\]
Likewise, for any $\mathsf{T}'\subseteq\mathsf{T}$, let us define
\[[\mathsf{T}']_\alpha=\{[S]_\alpha\mathrel{|}S\in\mathsf{T}'\}.\]
In particular, $[\mathsf{T}]_\alpha$ denotes the set of all equivalence classes of terms.  These will be used below in \S\ref{Consequences} to define formal typing statements.

The symmetry of head $\alpha$-conversion relies on the conditions we placed on the variables involved in the substitution.  However, for more general $\alpha$-conversion, these can be weakened as per \Cref{AlphaEqual}.

\begin{proposition}
    For any $R,S\in\mathsf{T}$ and $x,y\in\mathsf{V}$,
    \[y\notin\mathsf{F}(S)\quad\Rightarrow\quad\lambda xRS\equiv_\alpha\lambda yRS_{[y/x]}.\]
\end{proposition}

\begin{proof}
    See \cite[Lemma A1.8]{HindleySeldin2008}.
\end{proof}

Later in \S\ref{Inferences}, we will also need other basic syntactic properties of $\alpha$-conversion, such as
\begin{align*}
    S\equiv_\alpha S'\text{ and }T\equiv_\alpha T'\quad&\Rightarrow\quad S_{[T/x]}\equiv_\alpha S'_{[T'/x]},\\
    x=w\text{ or }x\notin\mathsf{F}(R)\quad&\Rightarrow\quad R_{[S/w][T/x]}\equiv_\alpha R_{[S_{[T/x]}/w]}\quad\text{and}\\
    x\neq w\notin\mathsf{F}(T)\quad&\Rightarrow\quad R_{[S/w][T/x]}\equiv_\alpha R_{[T/x][S_{[T/x]}/w]}
\end{align*}
(again see \cite[Appendix A]{HindleySeldin2008}, for example).  Indeed, this is why we still need $\alpha$-conversion even in such a bare bones type system.

\section{Abbreviations}\label{Abbreviations}

It is convenient to define various abbreviations for terms to make them easier to describe.  For starters, let us adopt the more standard convention of replacing $\beta$'s by parentheses, i.e.~for any $t_0,\ldots,t_n\in\mathsf{T}$, let $\beta^n$ denote $\beta$ repeated $n$ times and define
\[(t_0t_1\ldots t_{n-1}t_n):=\beta^nt_0t_1\ldots t_{n-1}t_n.\]
We even omit the parentheses if no ambiguity arises.

Next let us define $\pi:=\rho\lambda$ so, for all $x\in\mathsf{V}$ and $R,S\in\mathsf{T}$,
\[\llbracket\pi xRS\rrbracket=\llbracket\rho\lambda xRS\rrbracket=\prod\llbracket\lambda xRS\rrbracket=\prod_{r\in\llbracket R\rrbracket}\llbracket\lambda xRS\rrbracket(r)=\prod_{r\in\llbracket R\rrbracket}\llbracket S\rrbracket_{\langle r,x\rangle}.\]
So $\llbracket\pi xRS\rrbracket$ is a set of functions on $\llbracket R\rrbracket$ and accordingly we will also sometimes use the more suggestive notation $(x:R)\rightarrow S$, i.e.
\[(x:R)\rightarrow S\ :=\ \pi xRS.\]
As usual, multiple arrows associate to the right by default, i.e.
\[(x:R)\rightarrow(y:S)\rightarrow T:=(x:R)\rightarrow((y:S)\rightarrow T)=\pi xR\pi yST.\]

When $x$ is not free in $S$, we can even omit it, i.e.
\[R\rightarrow S\ :=\ \pi xRS,\quad\text{where }x\notin\mathsf{F}(S).\]
As before, it does not really matter which $x\in\mathsf{V}\setminus\mathsf{F}(S)$ we choose, although by default we could again choose the smallest, for example.  By \Cref{FreeProp}, $R\rightarrow S$ is interpreted as the set of all functions $\phi:\llbracket R\rrbracket\rightarrow\llbracket S\rrbracket$, i.e.~such that $\mathrm{dom}(\phi)=\llbracket R\rrbracket$ and $\mathrm{ran}(\phi)\subseteq\llbracket S\rrbracket$, as
\[\llbracket R\rightarrow S\rrbracket=\llbracket\pi xRS\rrbracket\ =\prod_{r\in\llbracket R\rrbracket}\llbracket S\rrbracket_{\langle r,x\rangle}=\prod_{r\in\llbracket R\rrbracket}\llbracket S\rrbracket\ =\ \llbracket S\rrbracket^{\llbracket R\rrbracket}.\]

The `propositions as types' philosophy views any set as an intuitionistic/constructive `proposition' with its elements being the `proofs' of the proposition.  So from this viewpoint, $R\rightarrow S$ is interpreted as a function that takes proofs of $\llbracket R\rrbracket$ to proofs of $\llbracket S\rrbracket$, which is the constructive view of logical implication, as per the Brouwer–Heyting–Kolmogorov interpretation of intuitionistic logic.

This viewpoint also dictates that the empty set has no proofs and thus represents the unique false proposition, while any non-empty set represents a true proposition.  Accordingly, we can view any function as encoding a relation which holds precisely when the result of applying the function is non-empty.  In other words, for any sets $R$ and $S$, any function $F\in S^R$ yields a corresponding subset/unary relation on $R$ defined by $\underline F=\{r\in R:F(r)\neq\emptyset\}$.  For example, if $R=Q\times Q$, for some set $Q$, then $\underline F$ is a binary relation on $Q$ such that, for all $p,q\in Q$,
\[p\mathrel{\underline F}q\qquad\Leftrightarrow\qquad F(p,q)\neq\emptyset.\]

To encode relations as functions in this way, it is convenient to distinguish a constant $*\in\mathsf{C}$ to denote a default family of sets we wish to consider as propositions.  We then define $\bot$ as an abbreviation for the term $\pi\mathsf{v}{*}\mathsf{v}$, which will be interpreted as $\emptyset$ as long as $\emptyset\in\llbracket *\rrbracket$, as then
\begin{equation}\label{bot}
    \llbracket\bot\rrbracket=\llbracket\pi \mathsf{v}{*}\mathsf{v}\rrbracket=\prod_{s\in\llbracket*\rrbracket}\llbracket \mathsf{v}\rrbracket_{\langle s,\mathsf{v}\rangle}=\emptyset\times\prod_{s\in\llbracket*\rrbracket\setminus\{\emptyset\}}\llbracket \mathsf{v}\rrbracket_{\langle s,\mathsf{v}\rangle}=\emptyset.
\end{equation}
So in this case $\bot$ is interpreted as `false'.  We then define $\top:=\bot\rightarrow\bot$, which is interpreted as a true proposition as long as $\emptyset,\{\emptyset\}\in\llbracket *\rrbracket$, as
\[\llbracket\top\rrbracket=\llbracket\bot\rightarrow\bot\rrbracket=\llbracket\bot\rrbracket^{\llbracket\bot\rrbracket}=\emptyset^\emptyset=\{\emptyset\}.\]

For any term $S$, it thus makes sense to define its \emph{negation} $\neg S$ by
\[\neg S\ :=\ S\rightarrow\bot.\]
As desired, this has the effect of switching true and false, as
\begin{equation}\label{negation}
    \llbracket\neg S\rrbracket=\llbracket S\rightarrow\bot\rrbracket=\llbracket\bot\rrbracket^{\llbracket S\rrbracket}=\emptyset^{\llbracket S\rrbracket}=\begin{cases}\ \,\emptyset&\text{if }\llbracket S\rrbracket\neq\emptyset\\\{\emptyset\}&\text{if }\llbracket S\rrbracket=\emptyset.\end{cases}
\end{equation}
As long as $\emptyset,\{\emptyset\}\in\llbracket *\rrbracket$, this shows that negation also has the effect turning an arbitrary set into one of our default propositions.  Let us further define the \emph{truncation} of any $S\in\mathsf{T}$ as its double negation $\neg\neg S$.  From the `propositions as types' perspective, truncation has the effect of identifying all proofs of the proposition in question.  Indeed,
\begin{equation}\label{truncation}
    \llbracket\neg\neg S\rrbracket=\begin{cases}\ \,\emptyset&\text{if }\llbracket\neg S\rrbracket\neq\emptyset\\\{\emptyset\}&\text{if }\llbracket\neg S\rrbracket=\emptyset\end{cases}=\begin{cases}\ \,\emptyset&\text{if }\llbracket S\rrbracket=\emptyset\\\{\emptyset\}&\text{if }\llbracket S\rrbracket\neq\emptyset.\end{cases}
\end{equation}

We can also define other logical operations that are consistent with the `propositions as types' viewpoint.  For example, we define the `and' operation $\wedge$ by
\[R\wedge S\ :=\ \neg(R\rightarrow\neg S).\]
Then $R\wedge S$ is `true' precisely when $R$ and $S$ are, as 
\[\llbracket R\wedge S\rrbracket\neq\emptyset\quad\Leftrightarrow\quad\llbracket R\rightarrow\neg S\rrbracket=\emptyset\quad\Leftrightarrow\quad\llbracket R\rrbracket\neq\emptyset\neq\llbracket S\rrbracket.\]
From the `propositions as types' perspective, we also see that $\pi$ acts like the universal quantifier.  Indeed, for all $x\in\mathsf{V}$ and $R,S\in\mathsf{T}$, we see that $\pi xRS$ is `true' precisely when $S$ is, for all values of $\llbracket R\rrbracket$, as
\[\llbracket\pi xRS\rrbracket\neq\emptyset\quad\Leftrightarrow\ \,\prod_{r\in\llbracket R\rrbracket}\llbracket S\rrbracket_{\langle r,x\rangle}\neq\emptyset\quad\Leftrightarrow\quad\llbracket S\rrbracket_{\langle r,x\rangle}\neq\emptyset,\text{ for all }r\in\llbracket R\rrbracket.\]
The only issue is that $\pi xRS$ may not be a default proposition, even when $S$ is, i.e.~we can have $\llbracket\pi xRS\rrbracket\notin\llbracket *\rrbracket\ni\llbracket S\rrbracket_{\langle r,x\rangle}$, for all $r\in\llbracket R\rrbracket$.  To deal with this, we can define $\forall$ as the truncation of $\pi$, i.e. $\forall:=\neg\neg\pi$, so that $\llbracket\forall xRS\rrbracket\in\llbracket *\rrbracket$, as long as $\emptyset,\{\emptyset\}\in\llbracket *\rrbracket$, and still
\[\llbracket\forall xRS\rrbracket\neq\emptyset\quad\Leftrightarrow\quad\llbracket S\rrbracket_{\langle r,x\rangle}\neq\emptyset,\text{ for all }r\in\llbracket R\rrbracket.\]

\begin{remark}\label{ImpredicativeRemark}
    Many logical systems based on dependent types take $\forall$ to be exactly the same as $\pi$ and introduce an inference rule to the effect that $\pi$-terms of propositions remain propositions, i.e.~of type $*$.  From a syntactic point of view, this `impredicativity' of $*$ might seem convenient, but unfortunately it is difficult if not impossible to justify semantically in any satisfactory way.  To show that such an impredicative $*$ is semantically sound in \cite{MiquelWerner2003} and \cite{Carneiro2019}, for example, they first have to split the language into propositional and non-propositional parts.  This breaks the key selling point of `propositions as types', namely that types and propositions can be treated in a uniform way.  These papers also rely on more complicated interpretations that depend not just on an assignment of sets to the variables and constants, but also on a given set of assumptions $\Gamma$ that can change the way terms are interpreted.  We feel it is better to keep the language unified and the semantics simple and just use truncation to simulate impredicativity by defining $\forall$ as $\neg\neg\pi$.
\end{remark}

Alternatively, logical operators (including $\forall$) can be defined as constant symbols satisfying appropriate specifications.  This is more in the spirit of our second logical system and is discussed below in \S\ref{Specifications}.

\section{Consequences}\label{Consequences}

Any interpretation $\llbracket\cdot\rrbracket$ yields a \emph{typing relation} $\in_{\llbracket\cdot\rrbracket}$ on $\mathsf{T}$ defined by
\[{\in_{\llbracket\cdot\rrbracket}}=\{\langle S,P\rangle\mathrel{|}S,P\in\mathsf{T}\text{ and }\llbracket S\rrbracket\in\llbracket P\rrbracket\}.\]
Alternatively, we can see $\in_{\llbracket\cdot\rrbracket}$ as a unary relation on \emph{$($typing$)$ statements} $\mathsf{S}$, which are formally ordered pairs of $\equiv_\alpha$-equivalence classes, i.e.
\[\tag{Statements}\mathsf{S}=[\mathsf{T}]_\alpha\times[\mathsf{T}]_{\alpha}.\]
As usual, for any $S,P\in\mathsf{T}$, we denote the corresponding statement by
\[(S:P)=\langle[S]_\alpha,[P]_\alpha\rangle\]
or just $S:P$ when convenient, reading this as `$S$ is of type $P$'.  Here $S$ is the \emph{subject} of the statement, while $P$ is the \emph{predicate} of the statement.  We invoke \Cref{AlphaEqual2} to view $\in_{\llbracket\cdot\rrbracket}$ as a unary relation on statements, which we again denote by $\llbracket\cdot\rrbracket$, i.e.~for all $S,P\in\mathsf{T}$,
\[\llbracket S:P\rrbracket\qquad\Leftrightarrow\qquad S\mathrel{\in_{\llbracket\cdot\rrbracket}}P\qquad\Leftrightarrow\qquad\llbracket S\rrbracket\in\llbracket P\rrbracket.\]
In this case we say $(S:P)$ is \emph{satisfied} by the given interpretation $\llbracket\cdot\rrbracket$.  We further extend this to sets of statements $\Gamma\subseteq\mathsf{S}$ by defining
\[\llbracket\Gamma\rrbracket\quad\Leftrightarrow\quad\llbracket S\rrbracket\in\llbracket P\rrbracket\text{ whenever }(S:P)\in\Gamma.\]
When $\llbracket\Gamma\rrbracket$ holds, we say that the interpretation $\llbracket\cdot\rrbracket$ is a \emph{model} for $\Gamma$.

\begin{definition}
The \emph{consequence relation} ${\vDash}\subseteq\mathcal{P}(\mathsf{S})\times\mathsf{S}$ is defined by
\[\Gamma\vDash(S:P)\quad\Leftrightarrow\quad\llbracket\Gamma\rrbracket\Rightarrow\llbracket S:P\rrbracket,\text{ for every interpretation }\llbracket\cdot\rrbracket.\]    
\end{definition}

So $\Gamma\vDash(S:P)$ means every model of $\Gamma$ is a model of $(S:P)$.  We could already view $\vDash$ as a kind of `semantic foundation' of mathematics, in the sense that any mathematical problem can be reduced to proving a particular instance of $\vDash$.  But of course this may not make the problem any easier, because as yet we have no obvious way of checking $\vDash$.  What we really want is some kind of syntactic equivalent of $\vDash$, or at least a good enough syntactic approximation of $\vDash$, obeying rules that can be checked mechanically by a computer, for example.  To find such an approximation, the idea is to investigate properties of $\vDash$ that we can later take as inference rules defining a syntactic `inference' relation ${\vdash}\subseteq{\vDash}$.  This is the goal of the present section.

The first basic properties to look at are those common to all classical logical systems.  First let us extend any ${\Vdash}\subseteq\mathcal{P}(\mathsf{S})\times\mathsf{S}$ to a binary relation on $\mathcal{P}(\mathsf{S})$ so that, for all $\Gamma,\Delta\subseteq\mathsf{S}$ and $S,P\in\mathsf{T}$,
\[\Gamma\Vdash\Delta\qquad\Leftrightarrow\qquad\Gamma\Vdash(S:P),\text{ for all }(S:P)\in\Delta.\]
We call ${\Vdash}$ a \emph{sequent} if this extension defines a preorder on $\mathcal{P}(\mathsf{S})$.  More explicitly this means that, for all $\Gamma,\Delta\subseteq\mathsf{S}$ and $S,P\in\mathsf{T}$,
\[\tag{Sequent}(S:P)\in\Gamma\text{ or }\Gamma\Vdash\Delta\Vdash(S:P)\quad\Rightarrow\quad\Gamma\Vdash (S:P).\]

\begin{proposition}\label{ConSeq}
    The consequence relation $\vDash$ is a sequent.
\end{proposition}

\begin{proof}
    If $(S:P)\in\Gamma$ and $\llbracket\Gamma\rrbracket$ holds then, in particular, $\llbracket S:P\rrbracket$ holds, showing that $\Gamma\vDash(S:P)$.  Likewise, if $\Gamma\Vdash\Delta\Vdash(S:P)$ and $\llbracket\Gamma\rrbracket$ holds then so does $\llbracket\Delta\rrbracket$ and hence $\llbracket S:P\rrbracket$, showing that $\Gamma\vDash(S:P)$.
\end{proof}

Let us call a relation ${\Vdash}\subseteq\mathcal{P}(\mathsf{S})\times\mathsf{S}$ \emph{monotone} if ${\supseteq}\circ{\Vdash}\subseteq{\Vdash}$, i.e.~if, for all $\Gamma,\Delta\subseteq\mathsf{S}$ and $S,P\in\mathsf{T}$,
\[\tag{Monotone}\Gamma\supseteq\Delta\Vdash(S:P)\quad\Rightarrow\quad\Gamma\Vdash(S:P).\]
We call ${\Vdash}\subseteq\mathcal{P}(\mathsf{S})\times\mathsf{S}$ \emph{reflexive} if, for all $S,P\in\mathsf{T}$,
\[\tag{Reflexive}(S:P)\Vdash(S:P).\]

\begin{proposition}\label{SequentMonoRefl}
    Every sequent ${\Vdash}$ is monotone and reflexive.
\end{proposition}

\begin{proof}
    Note $(S:P)\in\{(S:P)\}$ and hence $(S:P)\Vdash(S:P)$, proving reflexivity.  Also if $\Gamma\supseteq\Delta\Vdash(S:P)$ then $\Gamma\Vdash\Delta\Vdash(S:P)$ and hence $\Gamma\Vdash(S:P)$, proving monotonicity.
\end{proof}

The following is now immediate from \Cref{ConSeq,SequentMonoRefl}.

\begin{corollary}\label{ConMonoRefl}
    The consequence relation $\vDash$ is monotone and reflexive.
\end{corollary}

Next we have the following property of $\vDash$ relating to $\rho$-terms and $\beta$-terms.  For any $A,B\in\mathsf{S}$, let us abbreviate $\{A,B\}$ to $A,B$.

\begin{proposition}\label{AppProp}
    For any $F,P,R,S\in\mathsf{T}$ and $x\in\mathsf{V}$,
    \[(F:\pi xPR),(S:P)\vDash(FS:R_{[S/x]}).\]
\end{proposition}

\begin{proof}
    Assume that $\llbracket\cdot\rrbracket$ is a model for $(F:\pi xPR)$ and $(S:P)$.  This means $\llbracket F\rrbracket\in\prod_{p\in\llbracket P\rrbracket}\llbracket R\rrbracket_{\langle p,x\rangle}$ and $\llbracket S\rrbracket\in\llbracket P\rrbracket$ and hence
    \[\llbracket FS\rrbracket=\llbracket F\rrbracket(\llbracket S\rrbracket)\in\llbracket R\rrbracket_{\langle\llbracket S\rrbracket,x\rangle}=\llbracket R_{[S/x]}\rrbracket,\]
    by \eqref{Substitutivity}.  This means $\llbracket FS:R_{[S/x]}\rrbracket$, as required.
\end{proof}

We denote the free variables in any set of statements $\Gamma\subseteq\mathsf{S}$ by
\[\mathsf{F}(\Gamma)=\bigcup_{(S:P)\in\Gamma}\mathsf{F}(S)\cup\mathsf{F}(P)\]
We also write $\Gamma,(x:Q)$ as an abbreviation for $\Gamma\cup\{x:Q\}$.

\begin{proposition}\label{AbProp}
    For $\Gamma\subseteq\mathsf{S}$, $P,Q,S\in\mathsf{T}$ and $x\in\mathsf{V}\setminus(\mathsf{F}(\Gamma)\cup\mathsf{F}(Q))$,
    \[\Gamma,(x:Q)\vDash(S:P)\quad\Rightarrow\quad\Gamma\vDash(\lambda xQS:\pi xQP).\]
\end{proposition}

\begin{proof}
    Assume $\Gamma,(x:Q)\vDash(S:P)$ and take a model $\llbracket\cdot\rrbracket$ for $\Gamma$.  If $q\in\llbracket Q\rrbracket=\llbracket Q\rrbracket_{\langle q,x\rangle}$, as $x\notin\mathsf{F}(Q)$, then $\llbracket x:Q\rrbracket_{\langle q,x\rangle}$ and $\llbracket\Gamma\rrbracket_{\langle q,x\rangle}$, as $x\notin\mathsf{F}(\Gamma)$.  Thus $\llbracket S:P\rrbracket_{\langle q,x\rangle}$, i.e.~$\llbracket S\rrbracket_{\langle q,x\rangle}\in\llbracket P\rrbracket_{\langle q,x\rangle}$ and hence
    \[\llbracket\lambda xQS\rrbracket=\{\langle\llbracket S\rrbracket_{\langle q,x\rangle},q\rangle:q\in\llbracket Q\rrbracket\}\in\prod_{q\in\llbracket Q\rrbracket}\llbracket P\rrbracket_{\langle q,x\rangle}=\llbracket\pi xQP\rrbracket.\]
    This shows that $\llbracket\lambda xQS:\pi xQP\rrbracket$ and hence $\Gamma\vDash(\lambda xQS:\pi xQP)$.
\end{proof}

\section{Inferences}\label{Inferences}

The classic system of predicate logic has just two inference rules, modus ponens and universal generalisation (as in \cite[1.3.8 and 1.3.9]{ChangKeisler1990}, for example).  Viewing propositions as types, we see that these rules are analogous to the properties of $\vDash$ proved in \Cref{AppProp,AbProp}.  Indeed, when $x\notin\mathsf{F}(R)$, \Cref{AppProp} becomes
\[(F:P\rightarrow R),(S:P)\vDash(FS:R)\]
which turns into modus ponens once the subjects are erased.  Likewise, erasing the subjects in \Cref{AbProp} and replacing $\pi$ with $\forall$ yields a rule of predicate logic that follows from universal generalisation.

Accordingly, we take \Cref{AppProp,AbProp} as the sole inference rules defining our \emph{inference relation} ${\vdash}\subseteq\mathcal{P}(\mathsf{S})\times\mathsf{S}$.

\begin{definition}\label{vdashDef}
    Let ${\vdash}$ be the smallest sequent such that, for all $\Gamma\subseteq\mathsf{S}$, $F,P,Q,R,S\in\mathsf{T}$ and $x\in\mathsf{V}\setminus(\mathsf{F}(\Gamma)\cup\mathsf{F}(Q))$,
\begin{align}
    \label{App'}\tag{App$'$}(F:\pi xPR),(S:P)&\vdash(FS:R_{[S/x]})\quad\text{and}\\
    \label{Ab'}\tag{Ab}\Gamma,(x:Q)\vdash(S:P)\quad\Rightarrow\quad\Gamma&\vdash(\lambda xQS:\pi xQP).
\end{align}
\end{definition}

The following is now immediate from \Cref{ConSeq,AppProp,,AbProp}.

\begin{proposition}
    The inference relation is sound, i.e.~${\vdash}\subseteq{\vDash}$.
\end{proposition}

Many type systems in the literature restrict the left side of $\vdash$ to special finite sequences of statements known as `legal contexts'.  We make no such restriction here, although any particular instance of $\vdash$ is already determined by a finite subset of assumptions on the left, i.e.
\[\Gamma\vdash(S:P)\qquad\Rightarrow\qquad\exists\text{ finite }\Phi\subseteq\Gamma\ (\Phi\vdash(S:P)),\]
for any $\Gamma\subseteq\mathsf{S}$ and $S,P\in\mathsf{T}$.  To prove this and various other properties of the inference relation, it will be convenient to stratify $\vdash$ based on the number of times the inference rules need to be applied.

Accordingly, let ${\vdash_0}$ be the smallest monotone reflexive relation, i.e.
\[\Gamma\vdash_0(S:P)\qquad\Leftrightarrow\qquad(S:P)\in\Gamma.\]
Given $\vdash_n$, let $\vdash_{n+1}$ be the smallest monotone reflexive relation such that, for all $\Gamma\subseteq\mathsf{S}$, $F,P,Q,R,S\in\mathsf{T}$, $x\in\mathsf{V}$ and $y\in\mathsf{V}\setminus(\mathsf{F}(\Gamma)\cup\mathsf{F}(Q))$,
\begin{align}
    \label{Appn}\tag{App$_n$}\Gamma\vdash_n(F:\pi xPR),(S:P)\quad&\Rightarrow\quad\Gamma\vdash_{n+1}(FS:R_{[S/x]})\quad\text{and}\\
    \label{Abn}\tag{Ab$_n$}\Gamma,(y:Q)\vdash_n(S:P)\quad&\Rightarrow\quad\Gamma\vdash_{n+1}(\lambda yQS:\pi yQP).
\end{align}

Let us also define corresponding \emph{typing functions} $\lceil\cdot\rceil^\Gamma_n\in\mathcal{P}(\mathsf{T})^\mathsf{T}$, for all $\Gamma\subseteq\mathcal{P}(\mathsf{S})$ and $n\in\omega$, as follows.
First, for any $S\in\mathsf{T}$, let
\[\lceil S\rceil^\Gamma_0=\Gamma^{-1}([S]_\alpha)=\{P\mathrel{|}(S:P)\in\Gamma\}.\]
Once $\lceil\cdot\rceil^\Gamma_n$ has been defined, we define $\lceil\cdot\rceil^\Gamma_{n+1}$ by
\begin{align*}
    \hspace{-30pt}\lceil t\rceil^\Gamma_{n+1}&=\lceil t\rceil^\Gamma_0,\text{ if }t\in\mathsf{T}^0\cup\rho\mathsf{T},\\
    \hspace{-30pt}\lceil FS\rceil^\Gamma_{n+1}&=\lceil FS\rceil^\Gamma_0\cup\bigcup\{[R_{[S/x]}]_\alpha\mathrel{|}\exists P\in\lceil S\rceil^\Gamma_n\ (\pi xPR\in\lceil F\rceil^\Gamma_n)\}\text{ and}\\
    \hspace{-30pt}\lceil\lambda xQS\rceil^\Gamma_{n+1}&=\lceil\lambda xQS\rceil^\Gamma_0\cup\bigcup\{[\pi xQP]_\alpha\mathrel{|}\exists\Delta\subseteq\Gamma\ (x\notin\mathsf{F}(\Delta)\cup\mathsf{F}(Q)\text{ and }P\in\lceil S\rceil^{\Delta,(x:Q)}_n)\}.
\end{align*}

First we observe $\lceil\cdot\rceil^\Gamma_n$ is determined by its restriction to finite $\Gamma$.

\begin{proposition}\label{FiniteInference}
    For any $n\in\omega$ and $\Gamma\subseteq\mathsf{S}$,
    \[\lceil S\rceil^\Gamma_n=\bigcup\{\lceil S\rceil^\Phi_n\mathrel{|}\Phi\subseteq\Gamma\text{ is finite}\}.\]    
\end{proposition}

\begin{proof}
    First we claim that, for all $n\in\omega$, $\Gamma,\Delta\subseteq\mathsf{S}$ and $S\in\mathsf{T}$,
    \begin{equation}\label{TypeMonotonicity}
        \Delta\subseteq\Gamma\qquad\Rightarrow\qquad\lceil S\rceil^\Delta_n\subseteq\lceil S\rceil^\Gamma_n.
    \end{equation}
    Indeed, if $\Delta\subseteq\Gamma$ then $\lceil S\rceil^\Delta_0=\Delta^{-1}([S]_\alpha)\subseteq\Gamma^{-1}([S]_\alpha)=\lceil S\rceil^\Gamma_0$.  Also $\Delta'\subseteq\Delta$ then implies $\Delta'\subseteq\Gamma$ and hence $\lceil\lambda xQS\rceil^\Delta_n\subseteq\lceil\lambda xQS\rceil^\Gamma_n$, for all $\Delta'\subseteq\mathsf{S}$, $x\in\mathsf{V}$, $Q,S\in\mathsf{T}$ and $n\in\omega$.  Assuming \eqref{TypeMonotonicity} holds for $n$, it also then follows that $P\in\lceil S\rceil^\Delta_n$ and $\pi xPR\in\lceil F\rceil^\Delta_n$ implies $P\in\lceil S\rceil^\Gamma_n$ and $\pi xPR\in\lceil F\rceil^\Gamma_n$ and hence $\lceil FS\rceil^\Delta_{n+1}\subseteq\lceil FS\rceil^\Gamma_{n+1}$, for all $F,S\in\mathsf{T}$.  This proves \eqref{TypeMonotonicity} for $n+1$ and hence all $n\in\omega$, by induction.  In particular,
    \[\bigcup\{\lceil S\rceil^\Phi_n\mathrel{|}\Phi\subseteq\Gamma\text{ is finite}\}\subseteq\lceil S\rceil^\Gamma_n.\]

    For the reverse inclusion, we must show that
    \begin{equation}\label{TypeCompactness}
        P\in\lceil S\rceil^\Gamma_n\qquad\Rightarrow\qquad\exists\text{ finite }\Phi\subseteq\Gamma\ (P\in\lceil S\rceil^\Phi_n).
    \end{equation}
    By definition, $P\in\lceil S\rceil^\Gamma_0$ implies $(S:P)\in\Gamma$ and $P\in\lceil S\rceil^{(S:P)}_0$, proving \eqref{TypeCompactness} for $n=0$.  Now assume \eqref{TypeCompactness} for $n$ and take $P\in\lceil S\rceil^\Gamma_{n+1}$.  If $(S:P)\in\Gamma$ then again $P\in\lceil S\rceil^{(S:P)}_0\subseteq\lceil S\rceil^{(S:P)}_{n+1}$.  Otherwise $S$ is a $\beta$-term or a $\lambda$-term.  In the former case, we have $F,P',R,S'\in\mathsf{T}$ such that $S=FS'$, $P=R_{[S'/x]}$, $P'\in\lceil S'\rceil^\Gamma_n$ and $\pi xP'R\in\lceil F\rceil^{\Gamma}_n$.  Then we have finite $\Sigma,\Delta\subseteq\Gamma$ with $P'\in\lceil S'\rceil^\Sigma_n$ and $\pi xP'R\in\lceil F\rceil^\Delta_n$.  But then $\Phi=\Sigma\cup\Delta$ is also finite with $P'\in\lceil S'\rceil^\Phi_n$ and $\pi xP'R\in\lceil F\rceil^\Phi_n$, by \eqref{TypeMonotonicity}, and hence $P\in\lceil S\rceil^\Phi_{n+1}$.  The latter case is proved similarly, thus proving the result for $n+1$ and hence all $n\in\omega$, by induction.
\end{proof}

Next we observe that the typing functions do indeed characterise the corresponding inference relations.

\begin{proposition}\label{nTyping}
    For all $n\in\omega$, $\Gamma\subseteq\mathsf{S}$ and $S,P\in\mathsf{T}$,
    \[\Gamma\vdash_n(S:P)\qquad\Leftrightarrow\qquad P\in\lceil S\rceil^\Gamma_n.\]
\end{proposition}

\begin{proof}
    The $n=0$ case holds by definition.  Now assume the result for $n$.  For the $\Leftarrow$ part, first note that if $P\in\lceil S\rceil^\Gamma_0$ and hence $(S:P)\in\Gamma$ then $\Gamma\vdash_{n+1}(S:P)$, as $\vdash_{n+1}$ is monotone and reflexive.  In particular, $P\in\lceil t\rceil^\Gamma_{n+1}$ implies $\Gamma\vdash_n(t:P)$ when $t\in\mathsf{T}^0\cup\rho\mathsf{T}$.  Next note that if $P\in\lceil S\rceil^\Gamma_n$ and $\pi xPR\in\lceil F\rceil^\Gamma_n$ then, by the inductive hypothesis, $\Gamma\vdash_n(F:\pi xPR),(S:P)$ and hence $\Gamma\vdash_{n+1}(FS:R_{[S/x]})$.  Thus $P'\in\lceil FS\rceil^\Gamma_{n+1}$ implies $\Gamma\vdash_{n+1}(FS:P')$.  Finally note that if we have $\Delta\subseteq\Gamma$, $x\in\mathsf{V}\setminus(\mathsf{F}(\Delta)\cup\mathsf{F}(Q))$ and $P\in\lceil S\rceil^{\Delta,(x:Q)}_n$ then, again by the inductive hypothesis, $\Delta,(x:Q)\vdash_n(S:P)$ so $\Delta\vdash_{n+1}(\lambda xQS:\pi xQP)$ and hence $\Gamma\vdash_{n+1}(\lambda xQS:\pi xQP)$, as $\vdash_{n+1}$ is monotone.  This means $P'\in\lceil\lambda xQS\rceil^\Gamma_{n+1}$ implies $\Gamma\vdash_{n+1}(\lambda xQS:P')$.  Thus the $n+1$ case also holds and the $\Leftarrow$ part for all $n\in\omega$ follows by induction.

    Conversely, first note that a simple inductive argument shows that each $\lceil\cdot\rceil^\Gamma_n$ is $\alpha$-invariant in the sense that $S\equiv_\alpha S'$ implies $\lceil S\rceil^\Gamma_n=\lceil S'\rceil^\Gamma_n$ and $P'\equiv_\alpha P\in\lceil S\rceil^\Gamma_n$ implies $P'\in\lceil S\rceil^\Gamma_n$.  For each $n\in\omega$, we may thus temporarily define $\Gamma\vdash'_n(S:P)$ to mean $P\in\lceil S\rceil^\Gamma_n$.  As $\lceil S\rceil^\Gamma_0\subseteq\lceil S\rceil^\Gamma_n$, we immediately see that $\vdash'_n$ is reflexive.  Also $\vdash'_n$ is monotone, by \eqref{TypeMonotonicity}.  The definition of $\lceil\cdot\rceil^\Gamma_n$ also immediately shows that $\vdash'_n$ satisfies \eqref{Appn} and \eqref{Abn}, showing that ${\vdash_n}\subseteq{\vdash'_n}$ and hence ${\vdash_n}={\vdash'_n}$, as required.
\end{proof}

Now we observe $\vdash_n$ is invariant under substitution.  Here we define
\[\Gamma_{[T/x]}=\{(S_{[T/x]}:P_{[T/x]})\mathrel{|}(S:P)\in\Gamma\}.\]

\begin{proposition}\label{SubInv}
    For any $n\in\omega$, $\Gamma\subseteq\mathsf{S}$, $P,S,T\in\mathsf{T}$ and $x\in\mathsf{V}$,
    \[\Gamma\vdash_n(S:P)\qquad\Rightarrow\qquad\Gamma_{[T/x]}\vdash_n(S_{[T/x]}:P_{[T/x]}).\]
\end{proposition}

\begin{proof}
    First note $(S:P)\in\Gamma$ implies $(S_{[T/x]}:P_{[T/x]})\in\Gamma_{[T/x]}$ or, equivalently, $\Gamma\vdash_0(S:P)$ implies $\Gamma_{[T/x]}\vdash_0(S_{[T/x]}:P_{[T/x]})$.
    
    Now assume the result for $n$.  If $\Gamma\vdash_n(S:P)$ and $\Gamma\vdash_n(F:\pi xPR)$ then this means $\Gamma_{[T/x]}\vdash_n(S_{[T/x]}:P_{[T/x]})$.  Moreover, if $y=x$ or $x\notin\mathsf{F}(R)$ then $\Gamma_{[T/x]}\vdash_n(F_{[T/x]}:(\pi yPR)_{[T/x]})=(F_{[T/x]}:\pi yP_{[T/x]}R)$ so
    \[\Gamma_{[T/x]}\vdash_{n+1}(F_{[T/x]}S_{[T/x]}:R_{[S_{[T/x]}/y]})=((FS)_{[T/x]}:R_{[S/y][T/x]}).\]
    On the other hand, if $y\neq x\in\mathsf{F}(R)$ then we have $z\in\mathsf{V}\setminus(\mathsf{V}(RT))$ with $\Gamma_{[T/x]}\vdash_n(F_{[T/x]}:(\pi yPR)_{[T/x]})=(F_{[T/x]}:\pi zP_{[T/x]}R_{[z/y][T/x]})$ so
    \begin{align*}
        \Gamma_{[T/x]}\vdash_{n+1}(F_{[T/x]}S_{[T/x]}:R_{[z/y][T/x][S_{[T/x]}/z]})&=((FS)_{[T/x]}:R_{[z/y][S/z][T/x]})\\
        &=((FS)_{[T/x]}:R_{[S/y][T/x]}).
    \end{align*}
    Thus $\Gamma\vdash_{n+1}(FS:P')$ implies $\Gamma_{[T/x]}\vdash_{n+1}((FS)_{[T/x]}:P'_{[T/x]})$.

    Now say $\Delta$ is finite, $y\in\mathsf{V}\setminus(\mathsf{F}(\Delta)\cup\mathsf{F}(Q))$ and $\Delta,(y:Q)\vdash_n(S:P)$.  If $y\neq x$ then, for any $z\in\mathsf{V}\setminus(\mathsf{F}(\Delta)\cup\mathsf{F}(xPQST))$, it follows that $\Delta,(z:Q)\vdash_n(S_{[z/y]}:P_{[z/y]})$ so
    \[\Delta_{[T/x]},(z:Q_{[T/x]})\vdash_n(S_{[z/y][T/x]}:P_{[z/y][T/x]})\]
    and hence
    \begin{align*}
        \Delta_{[T/x]}\vdash_{n+1}&(\lambda zQ_{[T/x]}S_{[z/y][T/x]}:\pi zQ_{[T/x]}P_{[z/y][T/x]})\\
        =\ &((\lambda yQS)_{[T/x]}:(\pi yQP)_{[T/x]})
    \end{align*}
    On the other hand, if $y=x$ then we immediately see that
    \begin{align*}
        \Delta\vdash_{n+1}(\lambda yQS:\pi yQP)&=(\lambda yQ_{[T/x]}S:\pi yQ_{[T/x]}P)\\
        &=((\lambda yQS)_{[T/x]}:(\pi yQP)_{[T/x]}).
    \end{align*}
    Thus $\Gamma\vdash_{n+1}(\lambda yQS:P')$ implies $\Gamma_{[T/x]}\vdash_{n+1}((\lambda yQS)_{[T/x]}:P'_{[T/x]})$.  The general result now follows by induction.
\end{proof}

Using this we can show that ${\vdash_m}\circ{\vdash_n}\subseteq{\vdash_{m+n}}$, for all $m,n\in\omega$.

\begin{proposition}\label{vdashmn}
    For all $m,n\in\omega$, $\Gamma,\Delta\subseteq\mathsf{S}$ and $S,P\in\mathsf{T}$,
    \[\Gamma\vdash_m\Delta\vdash_n(S:P)\qquad\Rightarrow\qquad\Gamma\vdash_{m+n}(S:P).\]
\end{proposition}

\begin{proof}
    The proof is by induction.  First note that $\Gamma\vdash_m\Delta\vdash_0(S:P)$ means $\Gamma\vdash_m\Delta\ni(S:P)$, which immediately yields $\Gamma\vdash_m(S:P)$.  Now assume the result holds for $m$ and $n$ and say $\Gamma\vdash_m\Delta\vdash_{n+1}(S:P)$.  If $(S:P)\in\Delta$ then $\Gamma\vdash_m(S:P)$, as we just noted, and hence $\Gamma\vdash_{m+n}(S:P)$, as ${\vdash_m}\subseteq{\vdash_{m+n}}$.  Otherwise $S$ must be a $\beta$-term or a $\lambda$-term.  In the former case, \Cref{nTyping} yields $F,T,Q\in\mathsf{T}$ such that $S\equiv_\alpha FT$, $P\equiv_\alpha Q_{[T/x]}$ and $\Delta\vdash_n(F:\pi xRQ),(T:R)$.  From the inductive assumption, it then follows that $\Gamma\vdash_{m+n}(F:\pi xRQ),(T:R)$ and hence $\Gamma\vdash_{m+n+1}(FT:Q_{[T/x]})=(S:P)$, by \hyperref[Appn]{(App$_{m+n}$)}.  On the other hand, in the latter case \Cref{FiniteInference,nTyping} yield finite $\Sigma\subseteq\Delta$, $Q,S',P'\in\mathsf{T}$ and $x\in\mathsf{V}\setminus(\mathsf{F}(\Sigma)\cup\mathsf{F}(Q))$ such that $S\equiv_\alpha\lambda xQS'$, $P\equiv_\alpha\pi xQP'$ and $\Sigma,(x:Q)\vdash_n(S':P')$.  As $\Gamma\vdash_m\Delta\supseteq\Sigma$, we also have finite $\Phi\subseteq\Gamma$ with $\Phi\vdash_m\Sigma$.  Taking any $y\in\mathsf{V}\setminus(\mathsf{F}(\Phi)\cup\mathsf{F}(P'QS'))$, \Cref{SubInv} then yields $\Sigma,(y:Q)\vdash_n(S'_{[y/x]}:P'_{[y/x]})$.  As $\vdash_m$ is monotone and reflexive, $\Phi,(y:Q)\vdash_m\Sigma,(y:Q)$ and so the inductive assumption yields $\Phi,(x:Q)\vdash_{m+n}(S'_{[y/x]}:P'_{[y/x]})$.  Then \hyperref[Abn]{(Ab$_{m+n}$)} yields $\Gamma\supseteq\Phi\vdash_{m+n+1}(\lambda yQS'_{[y/x]}:\pi yQP'_{[y/x]})=(S:P)$.
\end{proof}

Now we can show $\vdash$ is the smallest monotone reflexive relation such that, for all $\Gamma\subseteq\mathsf{S}$, $F,Q,R,S,T\in\mathsf{T}$ and $x,y\in\mathsf{V}$ with $y\notin\mathsf{F}(\Gamma)\cup\mathsf{F}(Q)$,
\begin{align}
    \label{App}\tag{App}\Gamma\vdash(F:\pi xRS),(T:R)\quad&\Rightarrow\quad\Gamma\vdash(FT:S_{[T/x]})\quad\text{and}\\
    \label{Ab}\tag{Ab}\Gamma,(y:Q)\vdash(S:P)\quad&\Rightarrow\quad\Gamma\vdash(\lambda yQS:\pi yQP).  
\end{align}

\begin{corollary}
    We have ${\vdash}={\vdash_\omega}:=\bigcup_{n\in\omega}\vdash_n$.
\end{corollary}

\begin{proof}
    Certainly ${\vdash_0}\subseteq{\vdash}$.  Now assume that ${\vdash_n}\subseteq{\vdash}$.  Therefore if $\Gamma\vdash_n(F:\pi xRS),(T:R)$ then
    \[\Gamma\vdash(F:\pi xRS),(T:R)\vdash(FT:S_{[T/x]})\]
    so $\Gamma\vdash(FT:S_{[T/x]})$.  And if $y\notin\mathsf{F}(\Gamma)\cup\mathsf{F}(Q)$ and $\Gamma,(y:Q)\vdash_n(S:P)$ then $\Gamma,(y:Q)\vdash(S:P)$ so $\Gamma\vdash(\lambda yQS:\pi yQP)$.  \Cref{SequentMonoRefl} then yields ${\vdash_{n+1}}\subseteq{\vdash}$ so ${\vdash_\omega}\subseteq{\vdash}$, by induction.

    Conversely, ${\vdash_\omega}$ is transitive, by \eqref{vdashmn}.  Also $\vdash_\omega$ contains ${\vdash_0}$ and is thus a sequent.  But also $(F:\pi xRS),(T:R)\vdash_1(FT:S_{[T/x]})$ and, for any $y\notin\mathsf{F}(\Gamma)\cup\mathsf{F}(Q)$, we see that $\Gamma,(y:Q)\vdash_\omega(S:P)$ implies $\Gamma\vdash_\omega(\lambda yQS:\pi yQP)$.  This shows that ${\vdash}\subseteq{\vdash_\omega}$ as well.
\end{proof}

Similarly, let us define $\lceil\cdot\rceil^\Gamma\in\mathcal{P}(\mathsf{T})^\mathsf{T}$ so that, for all $\Gamma\subseteq\mathsf{S}$ and $S\in\mathsf{T}$,
\[\lceil S\rceil^\Gamma=\bigcup_{n\in\omega}\lceil S\rceil^\Gamma_n.\]
The defining properties of $\lceil\cdot\rceil^\Gamma_n$ then also apply to $\lceil\cdot\rceil^\Gamma$, i.e.
\begin{align*}
    \hspace{-30pt}\lceil t\rceil^\Gamma&=\lceil t\rceil^\Gamma_0,\text{ if }t\in\mathsf{T}^0\cup\rho\mathsf{T},\\
    \hspace{-30pt}\lceil FS\rceil^\Gamma&=\lceil FS\rceil^\Gamma_0\cup\bigcup\{[R_{[S/x]}]_\alpha\mathrel{|}\exists P\in\lceil S\rceil^\Gamma\ (\pi xPR\in\lceil F\rceil^\Gamma)\}\text{ and}\\
    \hspace{-30pt}\lceil\lambda xQS\rceil^\Gamma&=\lceil\lambda xQS\rceil^\Gamma_0\cup\bigcup\{[\pi xQP]_\alpha\mathrel{|}\exists\Delta\subseteq\Gamma\ (x\notin\mathsf{F}(\Delta)\cup\mathsf{F}(Q)\text{ and }P\in\lceil S\rceil^{\Delta,(x:Q)})\}.
\end{align*}
It follows that if $\Gamma$ is finite then, for all $S\in\mathsf{T}$, $\lceil S\rceil^\Gamma$ is also finite, modulo $\alpha$-conversion, and hence computable.  In other words, given $S$ and $\Gamma$, we can mechanically compute all the predicates/types $P$ such that $\Gamma\vdash(S:P)$.  More succinctly, we say that $\vdash$ has `type inference'.

If we also want unique typing (again up to $\alpha$-conversion) then we can restrict to certain subsets of $\mathsf{S}$ on the left.

\begin{definition}
    We call $\Gamma\subseteq\mathsf{S}$ a \emph{context} if $\Gamma^{-1}$ is a function with
    \[\mathrm{dom}(\Gamma^{-1})\subseteq[\mathsf{T}^0\cup\rho\mathsf{T}]_\alpha.\]
\end{definition}

Put another way, a context $\Gamma$ is just a set of type declarations for variables, constants and $\rho$-terms, at most one declaration for each term, i.e.~if $P,Q,S\in\mathsf{T}$ then $(S:P)\in\Gamma$ implies $S\in\mathsf{T}^0\cup\rho\mathsf{T}$ and also $(S:P),(S:Q)\in\Gamma$ implies $P\equiv_\alpha Q$.

\begin{proposition}\label{UniqueTyping}
    If $\Gamma$ is a context and $S,P,Q\in\mathsf{T}$ then
    \[\Gamma\vdash(S:P),(S:Q)\qquad\Rightarrow\qquad P\equiv_\alpha Q.\]
\end{proposition}

\begin{proof}
    As $\Gamma$ is a context, we know $\lceil S\rceil^\Gamma_0$ is empty unless $S\in\mathsf{T}^0\cup\rho\mathsf{T}$, in which case all terms in $\lceil S\rceil^\Gamma_0$ are $\alpha$-convertible.  Now assume that this is true for $n$.  Then $\lceil S\rceil^\Gamma_n$ and $\lceil F\rceil^\Gamma_n$ can contain at most one term up to $\alpha$-convertibility and so the same is true of $\lceil FS\rceil^\Gamma_{n+1}$.  On the other hand, say we have finite $\Delta,\Delta'\subseteq\Gamma$, $P,P',Q,S\in\mathsf{T}$ and $x\in\mathsf{V}$ such that $x\notin\mathsf{F}(\Delta)\cup\mathsf{F}(\Delta')\cup\mathsf{F}(Q)$, $P\in\lceil S\rceil^{\Delta,(x:Q)}_n$ and $P'\in\lceil S\rceil^{\Delta',(x:Q)}_n$.  Then $P,P'\in\lceil S\rceil^{\Delta\cup\Delta',(x:Q)}_n$ so $P\equiv_\alpha P'$ and hence $\pi xQP\equiv_\alpha\pi xQP'$, showing that $\lceil\lambda xQS\rceil^\Gamma_{n+1}$ also contains at most one term up to $\alpha$-convertibility.  This completes the proof for $\lceil\cdot\rceil^\Gamma_{n+1}$ and hence for $\vdash$ by induction.
\end{proof}

\section{Axioms}

This completes the set-up of our bare bones type system and the examination of its basic syntactic properties.  However, to use it as a foundational system, we also need some axioms.  These are statements that are satisfied by all the interpretations we wish to consider and which we should thus be free to add to $\Gamma$ when trying to verify a particular instance $\Gamma\vdash(S:P)$ of the inference relation.

\subsection{Equality}

So far we have not specified how to handle equality, even though this is of vital importance for any foundational system.  Coming from classical predicate logic, the temptation would be to add equality statements $(R\equiv S)$ as a primitive notion, in addition to the typing statements $(R:S)$ we already have.  One would then add an inference rule for substitution, e.g.~stating that $\Gamma,(R\equiv S)\vdash(T:U)$ always implies $\Gamma,(R\equiv S)\vdash(T':U')$, where $T'$ and $U'$ are obtained by replacing one or more instances of $R$ by $S$.  Indeed, if we were aiming to make $\vdash$ not just sound but also complete w.r.t. $\vDash$ in the traditional sense of ${\vDash}\subseteq{\vdash}$ then we would be forced to do something like this.

The problem is that these substitutions can break type inference and make even type checking an undecidable problem in general.  In other words, there could cease to be any mechanical way of verifying whether $\Gamma\vdash(R:S)$ holds as there will potentially be infinitely many ways of getting from $\Gamma$ to $(R:S)$ via sequences of substitutions.  Also, when looking at a typing statement $(S:P)$ from the `propositions as types' perspective, it is really only substitution in the predicate $P$ that we care about.  Indeed, from this viewpoint $S$ is just code for a proof of $P$ -- we do not care if the code is longer than necessary if that means it is easier to decode in a mechanical way.  And with this in mind, we would thus be happy with a weaker form of completeness, namely
\[\tag{Predicate Completeness}\Gamma\vDash(S:P)\quad\Rightarrow\quad\exists R\ (\Gamma\vdash(R:P)).\]

The standard solution is to instead take equality as a proposition, one which will allow us to do substitution in the predicate of a typing statement, at the cost of changing the subject.  First recall that we have a constant $*\in\mathsf{C}$ to denote our default propositions.  Given a term $R$ on which we want to define equality, we then take further constants $\mathsf{eq}_R^*,\mathsf{sub}_R^*\in\mathsf{C}$ and introduce the axioms $(\mathsf{eq}_R^*:R\rightarrow R\rightarrow *)$ and
\[(\mathsf{sub}_R^*:(t:R)\rightarrow(u:R)\rightarrow\mathsf{eq}_R^*tu\rightarrow(P:R\rightarrow *)\rightarrow Pt\rightarrow Pu).\]
So if $\Gamma$ contains these axioms and
\[\Gamma\vdash(g:\mathsf{eq}_R^*TU),(P:R\rightarrow*),(h:PT)\]
then the \eqref{App} rule yields $\Gamma\vdash(\mathsf{sub}_R^*TUgh:PU)$, showing that the $T$ in $PT$ can indeed be replaced by a $U$, at the cost of replacing the subject $h$ with the longer term $\mathsf{sub}_R^*TUgh$.

In particular, $P$ above could be of the form $\lambda xRS$.  Given that we will then have $\llbracket(\lambda xRS)U\rrbracket=\llbracket S_{[U/x]}\rrbracket$, we would naturally like to replace the predicate $(\lambda xRS)U$ by $S_{[U/x]}$ as well.  Again this can be achieved by adding a propositional form of $\beta$-reduction as another axiom.  Specifically, if $x\notin\mathsf{F}(R)$, $\mathsf{beta}_R^S\in\mathsf{C}$, $(\mathsf{beta}_R^S:(x:R)\rightarrow(\lambda xRS)x\rightarrow S)\in\Gamma$ and $\Gamma\vdash(h:(\lambda xRS)U),(U:R)$ then $\Gamma\vdash(\mathsf{beta}_R^SUh:S_{[U/x]})$.

\subsection{Polymorphism}\label{Polymorphism}

One problem with the above approach is that we have to distinguish new constants and their axioms for each new term $R$ that we wish to work with.  If $R$ happens to have free variables then another problem is that adding the axioms to $\Gamma$ will then prevent us from applying the \eqref{Ab'} rule with those variables.  To some extent this is mitigated by the monotonicity of $\vdash$, which allows us to add the relevant axioms to $\Gamma$ at the very end of a deduction.  However, this is only true when $\Gamma\vdash(S:P)$ can already be proved without including those axioms in $\Gamma$.

To avoid these problems, it is natural to make the axioms `polymorphic'.  Specifically, we first distinguish a constant $\square$ which is meant to denote some universe of sets on which we would like to define various notions.  The corresponding polymorphic notion of equality would then be defined by replacing the term $R$ with a variable $r$ and adding $(r:\square)\rightarrow$ to the beginning of the relevant predicates, e.g.
\[(\mathsf{eq}^*:(r:\square)\rightarrow r\rightarrow r\rightarrow *).\]
Likewise, to make $\beta$-reduction polymorphic in $R$, consider the axiom
\[(\mathsf{beta}^S:(r:\square)\rightarrow(x:r)\rightarrow(\lambda xrS)x\rightarrow S).\]

Note, however, that we can not make the $\beta$-reduction axiom polymorphic in $S$.  One might na\"ively consider the axiom
\begin{equation}\label{DoubleBeta}
    (\mathsf{beta}:(s:\square)\rightarrow(r:\square)\rightarrow(x:r)\rightarrow(\lambda xrs)x\rightarrow s),
\end{equation}
but when $x\in\mathsf{F}(S)$ (which is the situation we really care about), the substitution that would result from the \eqref{App} rule would end up changing the bound variable $x$.  In other words, given $\Gamma$ containing \eqref{DoubleBeta} and satisfying $\Gamma\vdash(S:\square)$, the \eqref{App} rule would result in $\Gamma\vdash(\mathsf{beta} S:(r:\square)\rightarrow(x:r)\rightarrow(\lambda yrS)x\rightarrow S)$, for some $y\neq x$.

This might be considered as one argument in favour of taking $\beta$-reduction as an inference rule rather than an axiom (as is usually done in pure type systems, for example) even though it could break the decidability of type checking, as mentioned above, or at least make it more difficult to prove.  A possibly more compelling argument would be that, in the presence of a $\beta$-reduction rule like
\begin{equation}\label{BasicBeta}
    (S:Q),(T:\lambda xQPS)\vdash(T:P_{[S/x]}),
\end{equation}
we can modify the \eqref{App} rule to avoid substitution, making it
\[(F:\pi xQR),(G:\rho F),(S:Q)\vdash(GS:FS).\]
Indeed, taking $F$ to be $\lambda xQP$ above and then applying \eqref{BasicBeta} would then yield the original \eqref{App} rule.

We would naturally like polymorphic notions to apply to sets of functions as well, i.e.~to terms of the form $\rho F$, for some $F\in\mathsf{T}$.  In an individual instance we could achieve this by simply adding $(\rho F:\square)$ as an axiom, so that our new polymorphic notions can indeed by applied on $\rho F$.  But again it would be nice if we did not have to add a new axiom like this for every new term $F$ that we wish to work with.  What we would really like is to be able to type $\rho$ as $(R\rightarrow\square)\rightarrow\square$, for all terms $R$.  But of course this does not make sense because syntactically $\rho$ is not even a properly formed term on its own.

\subsection{Logic}\label{Logic}

Before moving on, it is perhaps also worth mentioning the axioms that we do \emph{not} need, namely most of the logical tautologies that are usually taken as axioms in (0th order) sentential logic or (1st order) predicate logic.  For example, for any $R,S\in\mathsf{T}$, the tautology $R\rightarrow (R\rightarrow S)\rightarrow S$ is already witnessed by a term $T$, i.e.~we can construct $T\in\mathsf{T}$ such that $\emptyset\vdash(T:R\rightarrow (R\rightarrow S)\rightarrow S)$.  Indeed, taking any distinct $x,y\in\mathsf{V}\setminus\mathsf{F}(RS)$, we see that
\begin{align*}
    (x:R),(y:R\rightarrow S)&\vdash(yx:S),\text{ by \eqref{App}},\\
    \text{so}\quad(x:R)&\vdash(\lambda y(R\rightarrow S)(yx):(R\rightarrow S)\rightarrow S),\text{ by \eqref{Ab}},\\
    \text{and hence}\quad\emptyset&\vdash(\lambda xR\lambda y(R\rightarrow S)(yx):R\rightarrow(R\rightarrow S)\rightarrow S),
\end{align*}
again by \eqref{Ab}.  In particular, taking $S={\bot}$ yields $T\in\mathsf{T}$ with
\[\emptyset\vdash(T:R\rightarrow\neg\neg R).\]

However, in general we do not have any witness to the converse, i.e.~we may not have any $S\in\mathsf{T}$ with $\emptyset\vdash(S:\neg\neg R\rightarrow R)$.  In fact, if we are looking at the corresponding polymorphic term
\begin{equation}\label{ChoiceLEM}
    (x:\square)\rightarrow\neg\neg x\rightarrow x
\end{equation}
then we see that any $\phi\in\llbracket(x:\square)\rightarrow\neg\neg x\rightarrow x\rrbracket$ is a function such that $\phi(\emptyset)=\emptyset$ while $\phi(r)(\emptyset)\in r$, for all $r\in\llbracket\square\rrbracket\setminus\{\emptyset\}$.  Put another way, for each non-empty $r\in\llbracket\square\rrbracket$, the function $\phi$ is choosing an element of $r$.  So if we are considering $\square$ as denoting a universe of sets then \eqref{ChoiceLEM} is naturally viewed as the axiom of choice, while if we are instead taking $\square$ to denote a default set of propositions then \eqref{ChoiceLEM} is a form of the law of the excluded middle.  This could be viewed as a type theoretic manifestation of Diaconescu's theorem.

\subsection{Improvements}\label{Improvements}

Motivated by the remarks in \S\ref{Polymorphism}, it is natural to build a new type system with inbuilt polymorphism as follows.  From the outset, we distinguish a constant $\square$ to denote a universe of sets we wish to apply polymorphic definitions to.  We then distinguish another constant $\mathsf{p}$ to act as a product operator on the terms of $\square$, replacing all formal product terms in the inference rules with terms involving $\mathsf{p}$ instead.  To ensure that polymorphism also applies to product terms, we can type $\mathsf{p}$ as $(r:\square)\rightarrow(r\rightarrow\square)\rightarrow\square$.  The only issue is that the $\rightarrow$ symbols here should also be abbreviations for product terms, but unfolding them using the same product symbol $\mathsf{p}$ would only be valid if $\square:\square$, which is not satisfied by any interpretation (because sets are $\in$-well-founded, by the axiom of foundation in ZFC).  Thus we are naturally led to distinguishing another constant $\square'$ to denote another universe with $\square:\square'$ as well as another product operator constant $\mathsf{p}':(r:\square')\rightarrow(r\rightarrow\square')\rightarrow\square'$ which we can use to unfold the above arrows.  But to unfold these new arrows in turn requires another constant $\square''$ for an even higher universe, etc..  In the end we are led to distinguishing a countably infinite collection of constants to denote universes and products between them.  It is also natural then to take the lowest universe as a default set of propositions.

Our new system will also have (sub-)reduction statements and rules allowing us to replace the corresponding terms in typing statements.  We will usually take all $\beta$-sub-reduction statements as axioms and accordingly our \eqref{App} rule will be replaced with a substitution-free version, as outlined above.  The computation rules usually associated to inductive definitions can also be handled by adding them as (sub-)reduction statements.  Of course, the more of these statements we add, the more likely type-checking will become undecidable or difficult to prove.  But given the prevalence of reduction/computation rules resulting from `definitional equality' in the type systems underlying modern proof assistants, it seems that this is a price many are willing to pay.

\newpage
\part{A System with Polymorphic Product Operators}\label{Part2}

The preliminaries of our next system will be much like the first.  Thus we will be more terse and just review what is needed, pointing out the crucial differences -- see \Cref{Part1} for more details.

\section{The Language}

The language of our second system is just like the first minus the $\rho$-terms.  Specifically, the alphabet has five symbols denoted by
\[\mathsf{A}=\{\mathsf{v},\mathsf{c},{'},\beta,\lambda\}.\]
The strings we are interested in are defined by
\begin{align*}
    \tag{Constants}\mathsf{C}&::= \mathsf{c}\ |\ \mathsf{C}'\\
    \tag{Variables}\mathsf{V}&::= \mathsf{v}\ |\ \mathsf{V}'\\
    \tag{Terms}\mathsf{T}&::= \mathsf{C}\ |\ \mathsf{V}\ |\ \beta\mathsf{TT}\ |\ \lambda\mathsf{VTT}
\end{align*}
So $\mathsf{C}=\{\mathsf{c},\mathsf{c}',\mathsf{c}'',\ldots\}$, $\mathsf{V}=\{\mathsf{v},\mathsf{v}',\mathsf{v}'',\ldots\}$ and $\mathsf{T}$ is minimal with
\[\mathsf{C}\cup\mathsf{V}\cup\beta\mathsf{TT}\cup\lambda\mathsf{VTT}\subseteq\mathsf{T}.\]
Let $\mathsf{T}^0=\mathsf{C}\cup\mathsf{V}$ and $\mathsf{T}^{n+1}=\mathsf{T}^0\cup\beta\mathsf{T^nT^n}\cup\lambda\mathsf{VT^nT^n}$ so $\mathsf T=\bigcup_{n\in\omega}\mathsf T^n$.

As outlined in \S\ref{Improvements}, from the outset we distinguish a non-repeating sequence of constants $(\mathsf{u}_n)_{n\in\omega}\subseteq\mathsf{C}$ that we call \emph{sorts} to denote universes of domains of polymorphic functions.  We denote these sorts by
\[\tag{Sorts}\mathsf{U}=\{\mathsf{u}_n\mathrel{|}n\in\omega\}\subseteq\mathsf{C}.\]
For each $m,n\in\omega$, we further distinguish a unique constant $\mathsf{p}_m^n\in\mathsf{C}$ that is meant to denote the polymorphic product operator between the corresponding universes.  We denote these operators by
\[\tag{Operators}\mathsf{P}=\{\mathsf{p}^n_m\mathrel{|}m,n\in\omega\}\subseteq\mathsf{C}.\]

\begin{remark}
While the original systems of Barendregt's $\lambda$-cube (see \cite[\S13.E]{HindleySeldin2008}) distinguished only two sorts $*$ and $
\square$, modern pure type systems often distinguish a sequence of them as done here, such as Coquand and Paulin's Calculus of Inductive Constructions (see \cite{CoquandPaulin1990}) and Luo's Extended Calculus of Constructions (see \cite{Luo1994}).  We are just taking the next natural step and further distinguishing constants to use as polymorphic product operators on the sorts as well.
\end{remark}

Sequences of $\beta$'s at the start of a term will often be omitted as usual, i.e.~for all $R_0,\ldots,R_n\in\mathsf{T}$, we define the abbreviation
\[R_0\ldots R_n:=\beta^nR_0\ldots R_n.\]
Define $\mathsf{F}\in\mathcal{P}(\mathsf{V})^\mathsf{T}$ so that, for all $a\in\mathsf{C}$, $x\in\mathsf{V}$ and $R,S\in\mathsf{T}$,
\begin{align*}
    \mathsf{F}(a)&=\emptyset.\\
    \mathsf{F}(x)&=\{x\}.\\
    \mathsf{F}(RS)&=\mathsf{F}(R)\cup\mathsf{F}(S).\\
    \mathsf{F}(\lambda xRS)&=\mathsf{F}(R)\cup(\mathsf{F}(S)\setminus\{x\}).
\end{align*}
The variables in $\mathsf{F}(S)$ are said to be \emph{free}.  Abbreviations for product terms are then defined as follows
\begin{align*}
    \pi_m^nxRS&:=\mathsf{p}_m^nR\lambda xRS.\\
    (x:R)\rightarrow_m^nS&:=\pi_m^nxRS.\\
    R\rightarrow_m^nS&:=\pi_m^nxRS,\text{ where }x\notin\mathsf{F}(S).
\end{align*}
We will even drop the indices when they can be inferred from context, e.g.~writing $\mathsf{p}RG$, $\pi xRS$ and $R\rightarrow S$ for $\mathsf{p}_m^nRG$, $\pi_m^nxRS$ and $R\rightarrow_m^nS$.

For all $t\in\mathsf{T}^0$, $R,S,T\in\mathsf{T}$ and $x,y\in\mathsf{V}$, define
\begin{align*}
    t_{[T/x]}&:=\begin{cases}T&\text{if }t=x\\t&\text{if }t\in\mathsf{T}^0\setminus\{x\}.\end{cases}\\
    (RS)_{[T/x]}&:=R_{[T/x]}S_{[T/x]}.\\
    (\lambda yRS)_{[T/x]}&:=\begin{cases}\lambda yR_{[T/x]}S&\text{if }y=x\text{ or }x\notin\mathsf{F}(S)\\\lambda zR_{[T/x]}S_{[z/y][T/x]}&\text{otherwise, where }z\notin\mathsf{F}(\lambda yTS)\end{cases}
\end{align*}
As in \Cref{FreeChar}, as long as we choose $z$ appropriately,
\[S_{[T/x]}\neq S\qquad\Leftrightarrow\qquad T\neq x\in\mathsf{F}(S).\]

\section{Interpretations}

Denote the class of all sets by $\mathsf{Set}$.  We define \emph{interpretations} to be functions $\llbracket\cdot\rrbracket_\psi\in\mathsf{Set}^\mathsf{T}$, for $\psi\subseteq\mathsf{Set}\times(\mathsf{T}^0\setminus\mathsf{P})$ (where $\llbracket\cdot\rrbracket=\llbracket\cdot\rrbracket_\emptyset)$ such that, for all $t\in\mathsf{T}^0$, $R,S\in\mathsf{T}$, $x\in\mathsf{V}$ and $m,n\in\omega$,
\begin{align*}
    \llbracket t\rrbracket_\psi&=\begin{cases}\psi(t)&\text{if }t\in\mathrm{dom}(f)\\\llbracket t\rrbracket&\text{otherwise}.\end{cases}\\
    \llbracket RS\rrbracket&=\llbracket R\rrbracket(\llbracket S\rrbracket).\\
    \llbracket\lambda xRS\rrbracket&=\{\langle\llbracket S\rrbracket_{\langle r,x\rangle},r\rangle\mathrel{|}r\in\llbracket R\rrbracket\}.\\
    \llbracket\mathsf{p}_m^n\rrbracket&=\{\langle\{\langle\prod \phi,\phi\rangle\mathrel{|}\phi\in\llbracket\mathsf{u}_n\rrbracket^D\},D\rangle\mathrel{|}D\in\llbracket\mathsf{u}_m\rrbracket\}.
\end{align*}
So interpretations can take arbitrary values on $\mathsf{T}^0\setminus\mathsf{P}$ as before, but this time their values on the operators $\mathsf{P}$ are determined by their values on the sorts $\mathsf{U}$.  Specifically, for all $m,n\in\omega$, $\llbracket\mathsf{p}_m^n\rrbracket$ is the function on $\llbracket\mathsf{u}_m\rrbracket$ such that $\llbracket\mathsf{p}_m^n\rrbracket(D)$ is itself a function, for each $D\in\llbracket\mathsf{u}_m\rrbracket$, defined on each function $\phi\in\llbracket\mathsf{u}_n\rrbracket^D$ by
\[\llbracket\mathsf{p}_m^n\rrbracket(D)(\phi)=\prod\phi.\]

Unlike before, here we will be concerned not just with the actual interpretation of terms but whether the terms are \emph{well-formed} with respect to an interpretation, meaning that only functions are applied to other sets and only to those lying in their domain.  More precisely, let $\mathsf{Fun}$ denote the class of functions and, given any interpretation $\llbracket\cdot\rrbracket$, define $\mathsf{wf}_{\llbracket\cdot\rrbracket}\subseteq\mathsf{T}$ as the smallest set of terms such that $\mathsf{T}^0\subseteq\mathsf{wf}_{\llbracket\cdot\rrbracket}$,
\begin{align*}
    \{FS\mathrel{|}F,S\in\mathsf{wf}_{\llbracket\cdot\rrbracket},\ \llbracket F\rrbracket\in\mathsf{Fun}\text{ and }\llbracket S\rrbracket\in\mathrm{dom}\llbracket F\rrbracket\}&\subseteq\mathsf{wf}_{\llbracket\cdot\rrbracket}\quad\text{and}\\
    \{\lambda xRS\mathrel{|}R\in\mathsf{wf}_{\llbracket\cdot\rrbracket}\text{ and }\forall r\in\llbracket R\rrbracket\ (S\in\mathsf{wf}_{\llbracket\cdot\rrbracket_{\langle r,x\rangle}})\}&\subseteq\mathsf{wf}_{\llbracket\cdot\rrbracket}.
\end{align*}
Equivalently, define a unary relation $\llbracket\cdot\rrbracket^\mathsf{wf}$ on terms $\mathsf{T}$ such that
\begin{align*}
    \llbracket FS\rrbracket^\mathsf{wf}\quad&\Leftrightarrow\quad\llbracket F\rrbracket^\mathsf{wf},\llbracket S\rrbracket^\mathsf{wf},\ \llbracket F\rrbracket\in\mathsf{Fun}\text{ and }\llbracket S\rrbracket\in\mathrm{dom}\llbracket F\rrbracket,\\
    \llbracket\lambda xRS\rrbracket^\mathsf{wf}\quad&\Leftrightarrow\quad\llbracket R\rrbracket^\mathsf{wf}\text{ and }\forall r\in\llbracket R\rrbracket\ (\llbracket S\rrbracket_{\langle r,x\rangle}^\mathsf{wf})
\end{align*}
and $\llbracket t\rrbracket^\mathsf{wf}$, for all $t\in\mathsf{T}^0$.  Then $\mathsf{wf}_{\llbracket\cdot\rrbracket}=\{R\mathrel{|}\llbracket R\rrbracket^\mathsf{wf}\}$.

\begin{proposition}\label{wfSub}
    For all $S,T\in\mathsf{T}$ and $x\in\mathsf{V}$,
    \[\llbracket T\rrbracket^\mathsf{wf}\text{ and }\llbracket S\rrbracket_{\langle\llbracket T\rrbracket,x\rangle}^\mathsf{wf}\qquad\Rightarrow\qquad\llbracket S_{[T/x]}\rrbracket^\mathsf{wf}.\]
    The converse also holds when $x\in\mathsf{F}(S)$.
\end{proposition}

\begin{proof}
    This is certainly true for $S\in\mathsf{T}^0$ and can then be extended to all $\mathsf{T}^n$ by induction like in the proof of \Cref{SubstitutivityProp}.
\end{proof}

Given an interpretation $\llbracket\cdot\rrbracket$, we define a \emph{typing relation} $\in_{\llbracket\cdot\rrbracket}$ on $\mathsf{T}$ as before except that we also require the terms to be well-formed, i.e.
\[S\in_{\llbracket\cdot\rrbracket}P\qquad\Leftrightarrow\qquad\llbracket S\rrbracket^\mathsf{wf},\llbracket P\rrbracket^\mathsf{wf}\text{ and }\llbracket S\rrbracket\in\llbracket P\rrbracket.\]
Equivalently, $\in_{\llbracket\cdot\rrbracket}$ is a unary relation on \emph{typing statements} given by
\[\tag{Typing Statements}\mathsf{S}_:=\mathsf{T}\times\{:\}\times\mathsf{T}.\]
So this time typing statements are pairs of actual terms, not $\alpha$-conversion classes as before, together with a formal $:$ symbol added in the middle to distinguish them from (sub-)reduction statements below.  Again a particular typing statement $\langle S,:,P\rangle$ will usually be abbreviated to $S:P$ and read as `$S$ is of type $P$', referring to $S$ and $P$ as the \emph{subject} and \emph{predicate} of the typing statement.  Viewing $\in_{\llbracket\cdot\rrbracket}$ as a unary relation on typing statements, let us again denote it by $\llbracket\cdot\rrbracket$, i.e.~for all $S,P\in\mathsf{T}$,
\[\llbracket S:P\rrbracket\qquad\Leftrightarrow\qquad S\mathrel{\in_{\llbracket\cdot\rrbracket}}P\qquad\Leftrightarrow\qquad\llbracket S\rrbracket^\mathsf{wf},\llbracket P\rrbracket^\mathsf{wf}\text{ and }\llbracket S\rrbracket\in\llbracket P\rrbracket.\]
Again if $\llbracket S:P\rrbracket$ holds, we say that the typing statement $(S:P)$ is \emph{satisfied} by the interpretation $\llbracket\cdot\rrbracket$.

\begin{proposition}\label{RarrowSwf}
    For any interpretation $\llbracket\cdot\rrbracket$, $x\in\mathsf{V}$ and $R,S\in\mathsf{T}$,
    \[\llbracket(x:R)\rightarrow_m^nS\rrbracket^\mathsf{wf}\quad\Leftrightarrow\quad\llbracket R:\mathsf{u}_m\rrbracket\text{ and }\llbracket S:\mathsf{u}_n\rrbracket_{\langle r,x\rangle},\text{ for all }r\in\llbracket R\rrbracket.\]
\end{proposition}

\begin{proof}
    First note that $(x:R)\rightarrow_m^nS=\pi_m^nxRS=\mathsf{p}_m^nR(\lambda xRS)$ and
    \[\llbracket\mathsf{p}_m^nR(\lambda xRS)\rrbracket^\mathsf{wf}\ \Leftrightarrow\ \llbracket\mathsf{p}_m^nR\rrbracket^\mathsf{wf},\ \llbracket\lambda xRS\rrbracket^\mathsf{wf}\text{ and }\llbracket\lambda xRS\rrbracket\in\mathrm{dom}\llbracket\mathsf{p}_m^nR\rrbracket.\]
    But $\llbracket\mathsf{p}_m^nR\rrbracket^\mathsf{wf}$ means $\llbracket R\rrbracket^\mathsf{wf}$ and $\llbracket R\rrbracket\in\mathrm{dom}\llbracket\mathsf{p}_m^n\rrbracket=\llbracket\mathsf{u}_m\rrbracket$, i.e.~$\llbracket R:\mathsf{u}_m\rrbracket$.  In this case, $\llbracket\lambda xRS\rrbracket^\mathsf{wf}$ is equivalent to $\llbracket S\rrbracket_{\langle r,x\rangle}^\mathsf{wf}$, for all $r\in\llbracket R\rrbracket$.  Also then $\mathrm{dom}\llbracket\mathsf{p}_m^nR\rrbracket=\llbracket\mathsf{u}_n\rrbracket^{\llbracket R\rrbracket}$ and so $\llbracket\lambda xRS\rrbracket\in\mathrm{dom}\llbracket\mathsf{p}_m^nR\rrbracket$ precisely when $\llbracket S\rrbracket_{\langle r,x\rangle}\in\llbracket\mathsf{u}_n\rrbracket$ and hence $\llbracket S:\mathsf{u}_n\rrbracket_{\langle r,x\rangle}$, for all $r\in\llbracket R\rrbracket$.
\end{proof}

As mentioned in \S\ref{Improvements}, we are primarily interested in interpretations satisfying extra conditions corresponding to the `axioms' and `rules' of pure type systems (see \cite[\S5]{Barendregt1992}), conditions like $\llbracket\mathsf{u}_n\rrbracket\in\llbracket\mathsf{u}_{n+1}\rrbracket$ and $\llbracket\mathsf{p}_m^n\rrbracket(D)(\phi)\in\llbracket\mathsf{u}_{m\smallsmile n}\rrbracket$, for all $D\in\llbracket\mathsf{u}_m\rrbracket$ and $\phi\in\llbracket\mathsf{u}_n\rrbracket^D$, where $\smallsmile$ is some binary operation on $\omega$.  In other words, we are interested in interpretations $\llbracket\cdot\rrbracket$ satisfying typing statements like
\[(\mathsf{u}_n:\mathsf{u}_{n+1})\qquad\text{and}\qquad(\mathsf{p}_m^n:(x:\mathsf{u}_m)\rightarrow(x\rightarrow\mathsf{u}_n)\rightarrow\mathsf{u}_{m\smallsmile n}).\]

In keeping with idea that the sorts should denote `universes' in which we can specify various mathematical structures (see \S\ref{Specifications} below), we will also generally want each $\llbracket\mathsf{u}_n\rrbracket$ to satisfy as much of ZFC as possible.  If we are willing to accept the existence of infinitely many inaccessible cardinals then we can take each $\llbracket\mathsf{u}_n\rrbracket$ to satisfy all of ZFC, like in \cite{Carneiro2019}, for example.  However, we can avoid any large cardinal assumptions and still take each $\llbracket\mathsf{u}_n\rrbracket$ to satisfy ZFC$-$P (i.e.~ZFC without the power set axiom), which suffices for most specifications.  This is done by taking universes to be sets of the form $\mathsf{H}(\kappa)$, the family of all sets of hereditary cardinality less than $\kappa$ (see \cite{Kunen2011}).

The one exception here is the bottom sort $\mathsf{u}_0$ which we are taking to denote a default set of propositions, as suggested in \S\ref{Improvements}.  Accordingly, it is not so important for $\llbracket\mathsf{u}_0\rrbracket$ to satisfy any axioms of ZFC, rather we just want $\llbracket\mathsf{u}_0\rrbracket$ to include the default truth values $\emptyset$ (false) and $\{\emptyset\}$ (true), as in \S\ref{Abbreviations}.  If we want all `proofs' of a particular proposition to be identical, then we should also require each member of $\llbracket\mathsf{u}_0\rrbracket$ to contain at most one element.  On the other hand, if we want propositions to be identified with their truth values then we should make sure there is at most one non-empty member of $\llbracket\mathsf{u}_0\rrbracket$.  This all motivates \Cref{CanonicalInterpretation}.

First let us call a sequence $(\kappa_n)$ of regular cardinals a \emph{regular sequence} if, for all $n\in\omega$ and cardinals $\mu$ and $\nu$,
\[\mu<\kappa_n\quad\text{and}\quad\nu<\kappa_{n+1}\qquad\Rightarrow\qquad|\nu^\mu|<\kappa_{n+1}.\]
For example, we could take $\kappa_0=\aleph_0$ and just define $\kappa_{n+1}=|\kappa_n^{\kappa_n}|^+$, for all $n\in\omega$ (here $|S|$ denotes the cardinality of any set $S$ and $\mu^+$ denotes the successor of any cardinal $\mu$).  If the generalised continuum hypothesis holds, then we could also just take $\kappa_n=\aleph_n$, for all $n\in\omega$.  Or if there exists infinitely many inaccessible cardinals then we could simply let $(\kappa_n)$ be any strictly increasing sequence of them.

\begin{definition}\label{CanonicalInterpretation}
    We call an interpretation $\llbracket\cdot\rrbracket$ \emph{canonical} if we have a regular sequence $(\kappa_n)$ such that, for all $n\in\omega$,
    \[\{\emptyset,\{\emptyset\}\}\subseteq\llbracket\mathsf{u}_0\rrbracket\subseteq\mathsf{H}(\kappa_0)\qquad\text{and}\qquad\llbracket\mathsf{u}_n\rrbracket=\mathsf{H}(\kappa_n),\text{ for all }n\geq1.\]
    We also call an interpretation $\llbracket\cdot\rrbracket$
    \begin{enumerate}
        \item \emph{proof-irrelevant} if $|s|\leq1$, for all $s\in\llbracket\mathsf{u}_0\rrbracket$.
        \item \emph{propositionally extensional} if $|\llbracket\mathsf{u}_0\rrbracket\setminus\{\emptyset\}|\leq1$.
    \end{enumerate}
\end{definition}

In particular, a canonical interpretation $\llbracket\cdot\rrbracket$ is both proof-irrelevant and propositionally extensional precisely when $\llbracket\mathsf{u}_0\rrbracket=\{\emptyset,\{\emptyset\}\}$.

The first basic properties of canonical interpretations are as follows.

\begin{proposition}\label{umun}
    If $\llbracket\cdot\rrbracket$ is a canonical interpretation and $m,n\in\omega$,
    \[m<n\qquad\Leftrightarrow\qquad\llbracket\mathsf{u}_m:\mathsf{u}_n\rrbracket.\]
    Moreover, $\llbracket\mathsf{u}_n\rrbracket$ is a model of $\mathrm{ZFC}-\mathrm{P}$, for all $n\geq1$.
\end{proposition}

\begin{proof}
    First note $\mathsf{H}(\kappa)$ is a model of $\mathrm{ZFC}-\mathrm{P}$ for all uncountable regular $\kappa$, by \cite[II.2.1]{Kunen2011}, which immediately yields the last statement.

    If $m<n$ then certainly $\llbracket\mathsf{u}_m\rrbracket\subseteq\mathsf{H}(\kappa_m)\subseteq\mathsf{H}(\kappa_n)$.  By \cite[I.13.28]{Kunen2011}, $|\mathsf{H}(\kappa_m)|\leq\sup_{\mu,\nu<\kappa_m}|\nu^\mu|<\kappa_{m+1}\leq\kappa_n$ and hence $\llbracket\mathsf{u}_m\rrbracket\in\mathsf{H}(\kappa_n)=\llbracket\mathsf{u}_n\rrbracket$, by \cite[I.13.32]{Kunen2011}, i.e.~$\llbracket\mathsf{u}_m:\mathsf{u}_n\rrbracket$.  Conversely, if $\llbracket\mathsf{u}_m\rrbracket\in\llbracket\mathsf{u}_n\rrbracket$ then $m\neq n$, as sets are $\in$-well-founded, but also we can not have $m>n$ because that would imply $\llbracket\mathsf{u}_m\rrbracket=\mathsf{H}(\kappa_m)\supseteq\mathsf{H}(\kappa_n)\supseteq\llbracket\mathsf{u}_n\rrbracket\ni\llbracket\mathsf{u}_m\rrbracket$, again contradicting the fact sets are $\in$-well-founded.
\end{proof}

Canonicity also yields the following version of \Cref{RarrowSwf}.  Here $m\vee n$ denotes the maximum of $m$ and $n$.

\begin{proposition}\label{RarrowStype}
    If $\llbracket\cdot\rrbracket$ is canonical, $m,n\in\omega$, $x\in\mathsf{V}$ and $R,S\in\mathsf{T}$,
    \[\llbracket(x:R)\rightarrow_m^nS:\mathsf{u}_{(m+1)\vee n}\rrbracket\ \Leftrightarrow\ \llbracket R:\mathsf{u}_m\rrbracket\text{ and }\llbracket S:\mathsf{u}_n\rrbracket_{\langle r,x\rangle},\text{ for }r\in\llbracket R\rrbracket.\]
\end{proposition}

\begin{proof}
    By \Cref{RarrowSwf}, it suffices to show that the right side above implies $\llbracket(x:R)\rightarrow_m^nS\rrbracket\in\llbracket\mathsf{u}_{(m+1)\vee n}\rrbracket$.  But the right side implies $\llbracket R\rrbracket\in\llbracket\mathsf{u}_m\rrbracket\subseteq\mathsf{H}(\kappa_{(m+1)\vee n})$ and $\llbracket S\rrbracket_{\langle r,x\rangle}\in\llbracket\mathsf{u}_n\rrbracket\subseteq\mathsf{H}(\kappa_{(m+1)\vee n})$ and hence $\llbracket S\rrbracket_{\langle r,x\rangle}\subseteq\mathsf{H}(\kappa_{(m+1)\vee n})$, for all $r\in\llbracket R\rrbracket$.  Then by \cite[Lemma I.13.32]{Kunen2011}, $\prod_{r\in\llbracket R\rrbracket}\llbracket S\rrbracket_{\langle r,x\rangle}\subseteq\mathsf{H}(\kappa_{(m+1)\vee n})$ and it suffices to show $|\prod_{r\in\llbracket R\rrbracket}\llbracket S\rrbracket_{\langle r,x\rangle}|<\kappa_{(m+1)\vee n}$.  This follows because, for any function $\phi$ with $|D|<\kappa_m$ and $|\phi(d)|<\kappa_n$, for all $d\in D:=\mathrm{dom}(\phi)$,
    \begin{equation}\label{prodFcard}
        \Big|\prod\phi\Big|\leq|(\sup_{d\in D}|\phi(d)|)^{|D|}|<\kappa_{(m+1)\vee n}.
    \end{equation}
    Indeed, if $m>n$ then $\sup_{d\in D}|\phi(d)|\leq\kappa_n<\kappa_{m\vee n}$, while if $m\leq n$ then $\sup_{d\in D}|\phi(d)|<\kappa_n\leq\kappa_{m\vee n}$, by the regularity of $\kappa_n$.  So in either case $\sup_{d\in D}|\phi(d)|<\kappa_{m\vee n}\leq\kappa_{(m+1)\vee n}$ and $|D|<\kappa_m\leq\kappa_{((m+1)\vee n)-1}$ and hence $|(\sup_{d\in D}|\phi(d)|)^{|D|}|<\kappa_{(m+1)\vee n}$, as $(\kappa_j)$ is a regular sequence.
\end{proof}

In particular, if $\llbracket\cdot\rrbracket$ is canonical then, for all $m,n\in\omega$ and $R,S\in\mathsf{T}$,
\[\llbracket R:\mathsf{u}_m\rrbracket\quad\text{and}\quad\llbracket S:\mathsf{u}_n\rrbracket\qquad\Leftrightarrow\qquad\llbracket R\rightarrow_m^nS:\mathsf{u}_{(m+1)\vee n}\rrbracket.\]

We can now type the product operators $(\mathsf{p}_m^n)$ as follows.

\begin{proposition}
    If $\llbracket\cdot\rrbracket$ is a canonical interpretation and $m,n\in\omega$,
    \[\llbracket\mathsf{p}_m^n:(x:\mathsf{u}_m)\rightarrow(x\rightarrow\mathsf{u}_n)\rightarrow\mathsf{u}_{(m+1)\vee n}\rrbracket,\]
    specifically $\llbracket\mathsf{p}_m^n:(x:\mathsf{u}_m)\rightarrow_{m+1}^{(m\vee n)+2}(x\rightarrow_m^{n+1}\mathsf{u}_n)\rightarrow_{(m\vee n)+1}^{((m+1)\vee n)+1}\mathsf{u}_{(m+1)\vee n}\rrbracket$.
\end{proposition}

\begin{proof}
    For starters, we should show that
    \[\llbracket(x:\mathsf{u}_m)\rightarrow_{m+1}^{(m\vee n)+2}(x\rightarrow_m^{n+1}\mathsf{u}_n)\rightarrow_{(m\vee n)+1}^{((m+1)\vee n)+1}\mathsf{u}_{(m+1)\vee n}\rrbracket^\mathsf{wf}.\]
    By \Cref{umun}, $\llbracket\mathsf{u}_m:\mathsf{u}_{m+1}\rrbracket$ so, by \Cref{RarrowSwf}, we just have to show that, for all $r\in\llbracket\mathsf{u}_m\rrbracket$,
    \[\llbracket(x\rightarrow_m^{n+1}\mathsf{u}_n)\rightarrow_{(m\vee n)+1}^{((m+1)\vee n)+1}\mathsf{u}_{(m+1)\vee n}:\mathsf{u}_{(m\vee n)+2}\rrbracket_{\langle r,x\rangle}.\]
    Noting that $((m\vee n)+1+1)\vee(((m+1)\vee n)+1)=(m\vee n)+2$ and $\llbracket\mathsf{u}_{(m+1)\vee n}:\mathsf{u}_{((m+1)\vee n)+1}\rrbracket$, again by \Cref{umun}, we must just show
    \[\llbracket(x\rightarrow_m^{n+1}\mathsf{u}_n):\mathsf{u}_{(m\vee n)+1}\rrbracket_{\langle r,x\rangle},\]
    again by \Cref{RarrowStype}.  Noting that $(m+1)\vee(n+1)=(m\vee n)+1$ and $\llbracket x\rrbracket_{\langle r,x\rangle}=r\in\llbracket\mathsf{u}_m\rrbracket$ so $\llbracket x:\mathsf{u}_m\rrbracket_{\langle r,x\rangle}$, this amounts to showing that $\llbracket\mathsf{u}_n:\mathsf{u}_{n+1}\rrbracket$, which again follows from \Cref{umun}.

    Now it only remains to show that
    \begin{align*}
        \llbracket\mathsf{p}_m^n\rrbracket&\in\llbracket(x:\mathsf{u}_m)\rightarrow(x\rightarrow\mathsf{u}_n)\rightarrow\mathsf{u}_{(m+1)\vee n}\rrbracket\\
        &=\prod_{r\in\llbracket\mathsf{u}_m\rrbracket}\llbracket(x\rightarrow\mathsf{u}_n)\rightarrow\mathsf{u}_{(m+1)\vee n}\rrbracket_{\langle r,x\rangle}.
    \end{align*}
    As $\llbracket\cdot\rrbracket$ is an interpretation, $\mathrm{dom}\llbracket\mathsf{p}_m^n\rrbracket=\llbracket\mathsf{u}_m\rrbracket$ so we just have to show
    \[\llbracket\mathsf{p}_m^n\rrbracket(r)\in\llbracket(x\rightarrow\mathsf{u}_n)\rightarrow\mathsf{u}_{(m+1)\vee n}\rrbracket_{\langle r,x\rangle}=\llbracket\mathsf{u}_{(m+1)\vee n}\rrbracket^{\llbracket \mathsf{u}_n\rrbracket^r},\]
    for all $r\in\llbracket\mathsf{u}_m\rrbracket$.  But again we already know $\mathrm{dom}(\llbracket\mathsf{p}_m^n\rrbracket(r))=\llbracket\mathsf{u}_n\rrbracket^r$ because $\llbracket\cdot\rrbracket$ is an interpretation, so we just need to show that
    \[\prod \phi=\llbracket\mathsf{p}_m^n\rrbracket(r)(\phi)\in\llbracket\mathsf{u}_{(m+1)\vee n}\rrbracket,\]
    for all $\phi\in\llbracket\mathsf{u}_n\rrbracket^r$.  As $r\in\llbracket\mathsf{u}_m\rrbracket=\mathsf{H}(\kappa_m)$ and $\llbracket\mathsf{u}_n\rrbracket=\mathsf{H}(\kappa_n)$, we know that $\phi\subseteq\mathsf{H}(\kappa_{m\vee n})\subseteq\mathsf{H}(\kappa_{(m+1)\vee n})$ and hence $\prod \phi\subseteq\mathsf{H}(\kappa_{(m+1)\vee n})$.  Also $|\prod \phi|<\kappa_{(m+1)\vee n}$, by \eqref{prodFcard}, so $\prod\phi\in\mathsf{H}(\kappa_{(m+1)\vee n})=\llbracket\mathsf{u}_{(m+1)\vee n}\rrbracket$, thus completing the proof.
\end{proof}

In particular, if $\llbracket\cdot\rrbracket$ is canonical and $m\in\omega$ then
\[\llbracket\mathsf{p}_m^0:(x:\mathsf{u}_m)\rightarrow(x\rightarrow\mathsf{u}_0)\rightarrow\mathsf{u}_{m+1}\rrbracket.\]
So our canonical interpretations are never `impredicative', i.e.~we can not replace $m+1$ here with $0$, despite how tempting this might be.  Indeed, if $\llbracket\cdot\rrbracket$ is proof-irrelevant then $|\prod \phi|\leq1$ when $\mathrm{ran}(\phi)\subseteq\llbracket\mathsf{u}_0\rrbracket$ and so one might try and identify all singleton sets to enforce impredicativity.  However, this would necessarily mean forgetting the domain of each such $\phi$, which in turn would mean we could no longer tell which terms are well-formed with respect to $\llbracket\cdot\rrbracket$ (see \cite{MiquelWerner2003} for further discussion on this point).  In any case, as mentioned in \Cref{ImpredicativeRemark} and the comment after, we can easily simulate impredicativity in other ways, e.g.~by specifying universal quantification for $\llbracket\mathsf{u}_0\rrbracket$ separately from $\pi$, as discussed below in \S\ref{Specifications}.

\section{Reduction}

Every interpretation $\llbracket\cdot\rrbracket$ yields binary \emph{reduction} and \emph{sub-reduction} relations $\btright_{\llbracket\cdot\rrbracket}$ and $\triangleright_{\llbracket\cdot\rrbracket}$ on $\mathsf{T}$ defined by
\begin{align*}  R\btright\nolimits_{\llbracket\cdot\rrbracket}C\quad&\Leftrightarrow\quad\forall\psi\in\mathsf{Set}^\mathsf{V}\,(\llbracket R\rrbracket_\psi^\mathsf{wf}\Rightarrow(\llbracket C\rrbracket_\psi^\mathsf{wf}\text{ and }\llbracket R\rrbracket_\psi=\llbracket C\rrbracket_\psi)).\\
R\mathop{\triangleright_{\llbracket\cdot\rrbracket}}C\quad&\Leftrightarrow\quad\forall\psi\in\mathsf{Set}^\mathsf{V}\,(\llbracket R\rrbracket_\psi^\mathsf{wf}\Rightarrow(\llbracket C\rrbracket_\psi^\mathsf{wf}\text{ and }\llbracket R\rrbracket_\psi\subseteq\llbracket C\rrbracket_\psi)).
\end{align*}
Equivalently, we have unary relations on \emph{$($sub-$)$reduction statements}
\begin{align*}
    \tag{Reduction Statements}\mathsf{S}_{\btright}&=\mathsf{T}\times\{\btright\}\times\mathsf{T}.\\
    \tag{Sub-Reduction Statements}\mathsf{S}_\triangleright&=\mathsf{T}\times\{\triangleright\}\times\mathsf{T}.
\end{align*}
Particular (sub-)reduction statements $\langle R,{\btright},C\rangle$ and $\langle R,{\triangleright},C\rangle$ will usually be abbreviated to $R\btright C$ and $R\mathop{\triangleright}C$ and read as `$R$ (sub-)reduces to $C$', referring to $R$ and $C$ as the \emph{redex} and \emph{contractum} of the (sub-)reduction statement.  Viewing $\btright_{\llbracket\cdot\rrbracket}$ and $\triangleright_{\llbracket\cdot\rrbracket}$ as a unary relations on (sub-)reduction statements, we again denote them by $\llbracket\cdot\rrbracket$, i.e.
\begin{align*}
    \llbracket R\btright C\rrbracket\quad&\Leftrightarrow\quad\forall \psi\in\mathsf{Set}^\mathsf{V}\,(\llbracket R\rrbracket_\psi^\mathsf{wf}\Rightarrow(\llbracket C\rrbracket_\psi^\mathsf{wf}\text{ and }\llbracket R\rrbracket_\psi=\llbracket C\rrbracket_\psi)).\\
    \llbracket R\mathop{\triangleright}C\rrbracket\quad&\Leftrightarrow\quad\forall \psi\in\mathsf{Set}^\mathsf{V}\,(\llbracket R\rrbracket_\psi^\mathsf{wf}\Rightarrow(\llbracket C\rrbracket_\psi^\mathsf{wf}\text{ and }\llbracket R\rrbracket_\psi\subseteq\llbracket C\rrbracket_\psi)).
\end{align*}
So $\llbracket R\btright C\rrbracket$ and $\llbracket R\mathop{\triangleright}C\rrbracket$ mean that, no matter how we change the interpretation on the variables, if $R$ is well-formed with respect to the interpretation then so is $C$ and, in the case of $\btright$, their interpretations then agree or, in the case of $\triangleright$, the interpretation of $R$ is then contained in that of $C$.  In particular, $\llbracket R\btright C\rrbracket$ and $\llbracket R\triangleright C\rrbracket$ depend only on the values of $\llbracket\cdot\rrbracket$ on the constants $\mathsf{C}$.

For some trivial examples note, for any canonical interpretation $\llbracket\cdot\rrbracket$,
\begin{align*}
    m=n\qquad&\Leftrightarrow\qquad\llbracket\mathsf{u}_m\btright\mathsf{u}_n\rrbracket\quad\text{and}\\
    m\leq n\qquad&\Leftrightarrow\qquad\llbracket\mathsf{u}_m\mathop{\triangleright}\mathsf{u}_n\rrbracket.
\end{align*}
Another immediate example of reduction comes from $\alpha$-conversion.

\begin{proposition}\label{HeadAlphaModel}
    For any interpretation $\llbracket\cdot\rrbracket$, $R,S\in\mathsf{T}$ and $x,y\in\mathsf{V}$,
    \[y\notin\mathsf{F}(S)\qquad\Rightarrow\qquad\llbracket\lambda xRS\btright\lambda yRS_{[y/x]}\rrbracket.\]
\end{proposition}

\begin{proof}
    If $y\notin\mathsf{F}(S)$ then \Cref{wfSub} yields
    \[\llbracket S\rrbracket_{\langle r,x\rangle}^\mathsf{wf}\quad\Rightarrow\quad\llbracket S\rrbracket_{\langle r,y\rangle\langle\llbracket y\rrbracket_{\langle r,y\rangle},x\rangle}^\mathsf{wf}\quad\Rightarrow\quad\llbracket S_{[y/x]}\rrbracket_{\langle r,y\rangle}^\mathsf{wf}.\]
    So if $\llbracket\lambda xRS\rrbracket^\mathsf{wf}$ then $\llbracket S\rrbracket_{\langle r,x\rangle}^\mathsf{wf}$ and hence $\llbracket S_{[y/x]}\rrbracket_{\langle r,y\rangle}^\mathsf{wf}$, for all $r\in\llbracket R\rrbracket$, i.e.~$\llbracket\lambda yRS_{[y/x]}\rrbracket^\mathsf{wf}$.  Also $\llbracket\lambda xRS\rrbracket=\llbracket\lambda yRS_{[y/x]}\rrbracket$, as in \Cref{AlphaEqual}.  As $\llbracket\cdot\rrbracket$ was arbitrary, the same applies to $\llbracket\cdot\rrbracket_\phi$, for any $f\in\mathsf{Set}^\mathsf{V}$.
\end{proof}

Yet another example comes from $\beta$-reduction.

\begin{proposition}\label{HeadBetaModel}
    For any interpretation $\llbracket\cdot\rrbracket$, $R,S\in\mathsf{T}$ and $x\in\mathsf{V}$,
    \[\llbracket\beta\lambda xRST\btright S_{[T/x]}\rrbracket.\]
\end{proposition}

\begin{proof}
    We need to show $\llbracket\beta\lambda xRST\rrbracket^\mathsf{wf}$ always implies $\llbracket S_{[T/x]}\rrbracket^\mathsf{wf}$ and $\llbracket\beta\lambda xRST\rrbracket=\llbracket S_{[T/x]}\rrbracket$.  Now $\llbracket\beta\lambda xRST\rrbracket^\mathsf{wf}$ implies $\llbracket T\rrbracket^\mathsf{wf}$, $\llbracket\lambda xRS\rrbracket^\mathsf{wf}$ and $\llbracket T\rrbracket\in\mathrm{dom}\llbracket\lambda xRS\rrbracket=\llbracket R\rrbracket$.  Thus $\llbracket S\rrbracket_{\langle r,x\rangle}^\mathsf{wf}$, for all $r\in\llbracket R\rrbracket$, so $\llbracket S\rrbracket_{\langle\llbracket T\rrbracket,x\rangle}^\mathsf{wf}$ and hence $\llbracket S_{[T/x]}\rrbracket^\mathsf{wf}$, by \Cref{wfSub}.  Then \eqref{Substitutivity} yields
    \[\llbracket\beta\lambda xRST\rrbracket=\llbracket\lambda xRS\rrbracket(\llbracket T\rrbracket)=\llbracket S\rrbracket_{\langle\llbracket T\rrbracket,x\rangle}=\llbracket S_{[T/x]}\rrbracket.\qedhere\]
\end{proof}

The other standard reduction notion is $\eta$-reduction.  However, this is really an example of sub-reduction.

\begin{proposition}\label{EtaReduction}
    For any interpretation $\llbracket\cdot\rrbracket$, $R,F\in\mathsf{T}$ and $x\in\mathsf{V}$,
    \[x\notin\mathsf{F}(F)\qquad\Rightarrow\qquad\llbracket\lambda xR(Fx)\mathop{\triangleright}F\rrbracket.\]
\end{proposition}

\begin{proof}
    Assume $\llbracket\lambda xR(Fx)\rrbracket^\mathsf{wf}$ so $\llbracket Fx\rrbracket^\mathsf{wf}_{\langle s,x\rangle}$, for all $s\in\llbracket R\rrbracket$.  As $x\notin\mathsf{F}(F)$, it follows that $\llbracket F\rrbracket=\llbracket F\rrbracket_{\langle s,x\rangle}\in\mathsf{Fun}$ and $s\in\mathrm{dom}\llbracket F\rrbracket$, for all $s\in\llbracket R\rrbracket$, i.e.~$\llbracket R\rrbracket\subseteq\mathrm{dom}\llbracket F\rrbracket$.  Thus
    \begin{align*}
        \llbracket\lambda xR(Fx)\rrbracket&=\{\langle\llbracket Fx\rrbracket_{\langle s,x\rangle},s\rangle:s\in\llbracket R\rrbracket\}\\
        &\subseteq\{\langle\llbracket F\rrbracket(s),s\rangle:s\in\mathrm{dom}\llbracket F\rrbracket\}\\
        &=\llbracket F\rrbracket.
    \end{align*}
    As the interpretation was arbitrary, the same argument is valid for $\llbracket\cdot\rrbracket_\psi$, for any $\psi\in\mathsf{Set}^\mathsf{V}$, and so this shows that $\llbracket\lambda xR(Fx)\triangleright F\rrbracket$.
\end{proof}

Denote the set of reduction, sub-reduction and typing statements by
\[\tag{Statements}\mathsf{S}=\mathsf{S}_:\cup\mathsf{S}_\triangleright\cup\mathsf{S}_{\btright}.\]
Extending the terminology above, if $X\in\mathsf{S}$ and $\llbracket X\rrbracket$ holds then we say the statement $X$ is \emph{satisfied} by the interpretation $\llbracket\cdot\rrbracket$.  Likewise, if $\llbracket\cdot\rrbracket$ satisfies every statement in some given set $\Gamma\subseteq\mathsf{S}$ then we say $\llbracket\cdot\rrbracket$ is a \emph{model} for $\Gamma$ or again that $\llbracket\cdot\rrbracket$ \emph{satisfies} $\Gamma$.  Thus the results above are saying that every interpretation satisfies the \emph{head $\alpha$/$\beta$-reduction} and \emph{head $\eta$-sub-reduction} statements defined by
\begin{align*}
    {\btright\nolimits^\alpha}&=\{(\lambda xRS\btright\lambda yRS_{[y/x]})\mathrel{|}R,S\in\mathsf{T},x\in\mathsf{V}\setminus\mathsf{B}(S)\text{ and }y\in\mathsf{V}\setminus\mathsf{V}(S)\}.\\
    {\btright\nolimits^\beta}&=\{(\lambda xRS)T\btright S_{[T/x]})\mathrel{|}R,S,T\in\mathsf{T}\text{ and }x\in\mathsf{V}\}.\\
    {\triangleright^\eta}&=\{(\lambda xR(Fx)\mathop{\triangleright}F)\mathrel{|}R,F\in\mathsf{T}\text{ and }x\in\mathsf{V}\}.
\end{align*}

Like before, define the \emph{contextual closure} of any $\Gamma\subseteq\mathsf{T}\times\mathsf{Set}\times\mathsf{T}$ by
\[\Gamma^\mathsf{c}=\{\langle PRQ,\bullet,PCQ\rangle:PRQ,PCQ\in\mathsf{T},\langle R,\bullet,C\rangle\in\Gamma\text{ and }Q(0)\neq{'}\}.\]
Also define the transitive and reflexive closure of $\Gamma$ as before just leaving the extra symbol in the middle fixed.  We immediately see that any model for some $\Gamma\subseteq\mathsf{S}_{\btright}$ is also a model for its reflexive transitive contextual closure $\Gamma^\mathsf{ctr}$.  In particular, any interpretation satisfies the general \emph{$\alpha$/$\beta$-reduction} statements defined by
\[{\btright\nolimits_\alpha}={\btright\nolimits^{\alpha\mathsf{ctr}}}\qquad\text{and}\qquad{\btright\nolimits_\beta}={\btright\nolimits^{\beta\mathsf{ctr}}}.\]
However, a model for some $\Gamma\subseteq\mathsf{S}_\triangleright$ is also only a model for its reflexive transitive closure $\Gamma^\mathsf{rt}$, not its contextual closure $\Gamma^\mathsf{c}$ (as inclusions of sets need not be preserved by function application, for example).  In particular, interpretations may not satisfy $\triangleright^{\eta\mathsf{crt}}$, which is how one would usually define general $\eta$-sub-reduction statements, rather they only satisfy $\triangleright^{\eta\mathsf{rt}}$.

\section{Specifications}\label{Specifications}

Here we look at how statements can be used to specify certain mathematical structures.  So at the very least, any specification $\Gamma\subseteq\mathsf{S}$ should be \emph{satisfiable}, meaning it should have at least one model, preferably a canonical one.  But any canonical model for $\Gamma$ should also be isomorphic in an appropriate sense to the structure we want to specify.

The simplest example to start with is the specification for `false', i.e.~the empty set $\emptyset$.  First take $\bot\in\underline{\mathsf{C}}:=\mathsf{C}\setminus(\mathsf{U}\cup\mathsf{P})$ and a canonical interpretation $\llbracket\cdot\rrbracket$ with $\llbracket\bot\rrbracket=\emptyset$.  Certainly $\llbracket\bot:\mathsf{u}_0\rrbracket$, by the canonicity of $\llbracket\cdot\rrbracket$, but $\llbracket\mathsf{u}_0\rrbracket$ also contains non-empty sets so this is definitely not a complete specification.  What distinguishes $\emptyset$ is that we have a map from $\emptyset$ to $\prod_{r\in\llbracket u_0\rrbracket}r=\emptyset$, namely the empty map $\emptyset$ again.  Changing the interpretation at another constant $\lightning\in\underline{\mathsf{C}}$ so that $\llbracket\lightning\rrbracket=\emptyset$ too, this means $\llbracket\lightning:\bot\rightarrow_0^2(x:\mathsf{u}_0)\rightarrow_1^0 x\rrbracket$ or, leaving the indices implicit,
\[\llbracket\lightning:\bot\rightarrow(x:\mathsf{u}_0)\rightarrow x\rrbracket.\]
The completeness of this specification can now be stated as follows.

\begin{proposition}
    
    For any canonical $\llbracket\cdot\rrbracket$, ${\bot},{\lightning}\in\underline{\mathsf{C}}$ and $x\in\mathsf{V}$,
    \[\llbracket\bot:\mathsf{u}_0\rrbracket\quad\text{and}\quad\llbracket\lightning:\bot\rightarrow(x:\mathsf{u}_0)\rightarrow x\rrbracket\qquad\Rightarrow\qquad\llbracket\bot\rrbracket=\emptyset.\]
\end{proposition}

\begin{proof}
    Just note that $\llbracket(x:\mathsf{u}_0)\rightarrow x\rrbracket=\prod_{r\in\llbracket u_0\rrbracket}r=\emptyset$ because $\emptyset\in\llbracket\mathsf{u}_0\rrbracket$, as $\llbracket\cdot\rrbracket$ is canonical.  Thus $\llbracket\bot\rightarrow(x:\mathsf{u}_0)\rightarrow x\rrbracket=\emptyset^{\llbracket\bot\rrbracket}=\emptyset$ unless $\llbracket\bot\rrbracket=\emptyset$.  If $\llbracket\lightning\rrbracket\in\emptyset^{\llbracket\bot\rrbracket}$ then this is the only possibility.
\end{proof}

Note we do not even really need the assumption $\llbracket\bot:\mathsf{u}_0\rrbracket$ above, we are just adding it for clarity and so the implicit indices on the arrows in the statement $(\lightning:\bot\rightarrow(x:\mathsf{u}_0)\rightarrow x)$ can be inferred.

Another good example is equality.  This time take $a,\mathsf{eq}_a\in\underline{\mathsf{C}}$ and a canonical interpretation $\llbracket\cdot\rrbracket$ such that, for all $r,s\in\llbracket a\rrbracket$,
\[\llbracket\mathsf{eq}_a\rrbracket(r)(s)\neq\emptyset\qquad\Leftrightarrow\qquad r=s.\]
Thus we can change $\llbracket\cdot\rrbracket$ at another constant $\mathsf{rfl}_a\in\underline{\mathsf{C}}$ so $\llbracket\mathsf{rfl}_a\rrbracket$ takes each $r\in\llbracket a\rrbracket$ to an element of $\llbracket\mathsf{eq}_a\rrbracket(r)(r)$, i.e. $\llbracket\mathsf{rfl}_a\rrbracket\in\prod_{r\in\llbracket a\rrbracket}\llbracket\mathsf{eq}_a\rrbracket(r)(r)$.

So then $\llbracket\cdot\rrbracket$ is a canonical model for the typing statements $(a:\mathsf{u}_n)$, $(\mathsf{eq}_a:a\rightarrow a\rightarrow\mathsf{u}_0)$ and $(\mathsf{rfl}_a:(x:a)\rightarrow\mathsf{eq}_axx)$.  However, there are other models $
\llparenthesis\cdot\rrparenthesis$ of these statements where $\llparenthesis\mathsf{eq}_a\rrparenthesis(r)(s)$ is non-empty not just when $r=s$ but for some distinct $r,s\in\llbracket a\rrbracket$ as well.  To complete the specification, we can use the substitution property of equality, specifically the fact that, for any $\phi:\llbracket a\rrbracket\rightarrow\llbracket\mathsf{u}_0\rrbracket$, if $r=s$ and $\phi(r)\neq\emptyset$ then $\phi(s)\neq\emptyset$ too.  So we can change $\llbracket\cdot\rrbracket$ at another constant $\mathsf{sub}_a\in\underline{\mathsf{C}}$ so that $\llbracket\mathsf{sub}_a\rrbracket\in\prod_{r\in\llbracket a\rrbracket}\prod_{s\in\llbracket a\rrbracket}(\prod_{\phi\in\llbracket\mathsf{u}_0\rrbracket^{\llbracket a\rrbracket}}\phi(s)^{\phi(r)})^{\llbracket\mathsf{eq}_a\rrbracket(r)(s)}$ and thus $\llbracket\cdot\rrbracket$ is a canonical model for the typing statement
\[(\mathsf{sub}_a:(x:a)\rightarrow(y:a)\rightarrow\mathsf{eq}_axy\rightarrow(p:a\rightarrow\mathsf{u}_0)\rightarrow px\rightarrow py).\]
This completes the specification, as shown by the following.

\begin{proposition}\label{EqSpec}
    Given $n\in\omega$, $x,y,p\in\mathsf{V}$ and $a,\mathsf{eq}_a,\mathsf{rfl}_a\in\underline{\mathsf{C}}$, if
\begin{align*}
    &(a:\mathsf{u}_n),(\mathsf{eq}_a:a\rightarrow a\rightarrow\mathsf{u}_0),(\mathsf{rfl}_a:(x:a)\rightarrow\mathsf{eq}_axx)\quad\text{and}\\
    &(\mathsf{sub}_a:(x:a)\rightarrow(y:a)\rightarrow\mathsf{eq}_axy\rightarrow(p:a\rightarrow\mathsf{u}_0)\rightarrow px\rightarrow py)
\end{align*}
    have a canonical model $\llbracket\cdot\rrbracket$ then, for all $r,s\in\llbracket a\rrbracket$,
    \[\llbracket\mathsf{eq}_a\rrbracket(r)(s)\neq\emptyset\qquad\Leftrightarrow\qquad r=s.\qedhere\]
\end{proposition}

\begin{proof}
    Note $\llbracket\mathsf{rfl}_a:(x:a)\rightarrow\mathsf{eq}_axx\rrbracket$ means $\llbracket\mathsf{rfl}_a\rrbracket\in\prod_{r\in\llbracket a\rrbracket}\llbracket\mathsf{eq}_a\rrbracket(r)(r)$ and hence $\llbracket\mathsf{eq}_a\rrbracket(r)(r)\neq\emptyset$, for all $r\in\llbracket a\rrbracket$.

    Conversely, say we had distinct $r,s\in\llbracket a\rrbracket$ and $t\in\llbracket\mathsf{eq}_a\rrbracket(r)(s)$.  Take $\phi\in\llbracket\mathsf{u}_0\rrbracket^{\llbracket a\rrbracket}$ with $\phi(r)\neq\emptyset=\phi(s)$, so we have some $u\in\phi(r)$.  Then $\llbracket\mathsf{sub}_a:(x:a)\rightarrow(y:a)\rightarrow\mathsf{eq}_axy\rightarrow(p:a\rightarrow\mathsf{u}_0)\rightarrow px\rightarrow py\rrbracket$ implies
    \[\llbracket\mathsf{sub}_a\rrbracket(r)(s)(t)(\phi)(u)\in\phi(s)=\emptyset,\]
    a contradiction, completing the proof.
\end{proof}

As noted in the previous section, $\llbracket R\btright S\rrbracket$ always implies $\llbracket (R\btright S)^\mathsf{c}\rrbracket$ and hence $\llbracket(R\mathop{\triangleright}S)^\mathsf{c}\rrbracket$.  When $\llbracket\cdot\rrbracket$ satisfies the above specification, we can also reverse this and derive equality from the contextual closure of a sub-reduction statement.

\begin{proposition}\label{SubReductionEq}
    If $\llbracket\cdot\rrbracket$ satisfies the typing statements in \Cref{EqSpec} then, for all $R,S\in\mathsf{T}$,
    \[\llbracket R:a\rrbracket\quad\text{and}\quad\llbracket(R\mathop{\triangleright}S)^\mathsf{c}\rrbracket\qquad\Rightarrow\qquad\llbracket R\rrbracket=\llbracket S\rrbracket.\]
\end{proposition}

\begin{proof}
If $\llbracket R:a\rrbracket$ then $\llbracket\mathsf{eq}_aRR\rrbracket^\mathsf{wf}$ so $\llbracket(R\mathop{\triangleright}S)^\mathsf{c}\rrbracket$ implies
\[\emptyset\neq\llbracket\mathsf{eq}_aRR\rrbracket\subseteq\llbracket\mathsf{eq}_aRS\rrbracket\]
and hence $\llbracket R\rrbracket=\llbracket S\rrbracket$.    
\end{proof}

The above specifications were for elements of/functions to our default set of propositions $\llbracket\mathsf{u}_0\rrbracket$ where we only care about their truth values, i.e.~whether the sets in $\llbracket\mathsf{u}_0\rrbracket$ are empty or not.  For specifications in higher universes, we care more about the precise sets and functions we are specifying.  This is where we can use (sub-)reduction statements.

To take a basic example, let us consider the usual binary Cartesian product.  First take $a,b\in\underline{\mathsf{C}}$ and $n\geq1$ and let $\llbracket\cdot\rrbracket$ be a canonical model of $(a:\mathsf{u}_n)$ and $(b:\mathsf{u}_n)$.  We can then change the interpretation at another constant $\mathsf{pr}_{a,b}\in\underline{\mathsf{C}}$ to ensure that
\[\llbracket\mathsf{pr}_{a,b}\rrbracket=\llbracket a\rrbracket\times\llbracket b\rrbracket\]
and hence $\llbracket\mathsf{pr}_{a,b}\rrbracket\in\llbracket\mathsf{u}_n\rrbracket$, by the canonicity of $\llbracket\cdot\rrbracket$.  To reflect the fact elements of $\llbracket a\rrbracket\times\llbracket b\rrbracket$ are made from those of $\llbracket a\rrbracket$ and $\llbracket b\rrbracket$, we can change the interpretation at another constant $\mathsf{mk}_{a,b}\in\underline{\mathsf{C}}$ to ensure that $\llbracket\mathsf{mk}_{a,b}\rrbracket\in(\llbracket \mathsf{pr}_{a,b}\rrbracket^{\llbracket b\rrbracket})^{\llbracket a\rrbracket}$ and, for all $r\in\llbracket a\rrbracket$ and $s\in\llbracket b\rrbracket$,
\[\llbracket\mathsf{mk}_{a,b}\rrbracket(r)(s)=\langle r,s\rangle.\]
This means $\llbracket\mathsf{mk}_{a,b}:a\rightarrow b\rightarrow\mathsf{pr}_{a,b}\rrbracket$, but of course we could change the interpretation at $\mathsf{mk}_{a,b}$ in many other ways and still satisfy this typing statement, as there are many other functions in $(\llbracket\mathsf{pr}_{a,b}\rrbracket^{\llbracket b\rrbracket})^{\llbracket a\rrbracket}$.

As a first step to completing the specification, we can use the key property of $\llbracket a\rrbracket\times\llbracket b\rrbracket$ that, for any function $\theta\in(\mathsf{Set}^{\llbracket b\rrbracket})^{\llbracket a\rrbracket}$, we have a (unique) function $\theta^\times\in\mathsf{Set}^{\llbracket a\rrbracket\times\llbracket b\rrbracket}$ such that $\theta^\times(r,s)=\theta(r)(s)$, for all $r\in\llbracket a\rrbracket$ and $s\in\llbracket b\rrbracket$.  We can thus change the interpretation at another constant $\mathsf{rec}_{a,b}\in\underline{\mathsf{C}}$ to ensure that $\llbracket\mathsf{rec}_{a,b}\rrbracket$ is a function such that, for all $\phi\in\llbracket\mathsf{u}_n\rrbracket^{\llbracket\mathsf{pr}_{a,b}\rrbracket}$ and $\theta\in\prod\limits_{r\in\llbracket a\rrbracket}\prod\limits_{s\in\llbracket b\rrbracket}\phi(r,s)$,
\[\llbracket\mathsf{rec}_{a,b}\rrbracket(\phi)(\theta)=\theta^\times.\]
As long as we also make $\llbracket\mathsf{rec}_{a,b}\rrbracket$ undefined elsewhere, this means \[\llbracket\mathsf{rec}_{a,b}\rrbracket\in\prod_{\phi\in\llbracket\mathsf{u}_n\rrbracket^{\llbracket\mathsf{pr}_{a,b}\rrbracket}}\Big(\prod_{t\in\llbracket\mathsf{pr}_{a,b}\rrbracket}\phi(t)\Big)^{\prod\limits_{r\in\llbracket a\rrbracket}\prod\limits_{s\in\llbracket b\rrbracket}\phi(r,s)}.\]
In particular, $\llbracket\cdot\rrbracket$ also satisfies the typing statement
\[\mathsf{rec}_{a,b}:(f:\mathsf{pr}_{a,b}\rightarrow\mathsf{u}_n)\rightarrow((x:a)\rightarrow (y:b)\rightarrow f(\mathsf{mk}_{a,b}xy))\rightarrow(z:\mathsf{pr}_{a,b})\rightarrow fz.\]
But again there are plenty of other functions we could assign to $\mathsf{rec}_{a,b}$ that still satisfy this statement so our specification is still incomplete.

To really complete it note that, for any $\psi\in\mathsf{Set}^\mathsf{V}$ and $f,g,x,y\in\mathsf{V}$,
\begin{align*}
    \llbracket\mathsf{rec}_{a,b}fg(\mathsf{mk}_{a,b}xy)\rrbracket^\mathsf{wf}_\psi\quad\Rightarrow\quad\llbracket\mathsf{rec}_{a,b}fg(\mathsf{mk}_{a,b}xy)\rrbracket_\psi&=\llbracket g\rrbracket_\psi^\times(\llbracket x\rrbracket_\psi,\llbracket y\rrbracket_\psi)\\
    &=\llbracket g\rrbracket_\psi(\llbracket x\rrbracket_\psi)(\llbracket y\rrbracket_\psi)\\
    &=\llbracket gxy\rrbracket_\psi.
\end{align*}
In other words, $\llbracket\cdot\rrbracket$ satisfies the reduction statement
\[\mathsf{rec}_{a,b}fg(\mathsf{mk}_{a,b}xy)\btright gxy.\]

Now we can show that these statements completely specify the product in the sense that any canonical interpretation $\llbracket\cdot\rrbracket$ satisfying them gives us a natural bijection from $\llbracket\mathsf{pr}_{a,b}\rrbracket$ onto $\llbracket a\rrbracket\times\llbracket b\rrbracket$.

\begin{proposition}\label{CartesianProduct}
    Given $n\geq1$, $f,g,x,y,z\in\mathsf{V}$ as well as constants $a,b,\mathsf{pr}_{a,b},\mathsf{mk}_{a,b},\mathsf{rec}_{a,b}\in\underline{\mathsf{C}}$, if $\llbracket\cdot\rrbracket$ is a canonical model of
    \begin{align*}
        &(a,b,\mathsf{pr}_{a,b}:\mathsf{u}_n),(\mathsf{mk}_{a,b}:a\rightarrow b\rightarrow\mathsf{pr}_{a,b}),(\mathsf{rec}_{a,b}fg(\mathsf{mk}_{a,b}xy)\btright gxy)\quad\text{and}\\
        &(\mathsf{rec}_{a,b}:(f:\mathsf{pr}_{a,b}\rightarrow\mathsf{u}_n)\rightarrow((x:a)\rightarrow (y:b)\rightarrow f(\mathsf{mk}_{a,b}xy))\rightarrow(z:\mathsf{pr}_{a,b})\rightarrow fz)
    \end{align*}
    then $\langle r,s\rangle\mapsto\llbracket\mathsf{mk}_{a,b}\rrbracket(r)(s)$ is a bijection from $\llbracket a\rrbracket\times\llbracket b\rrbracket$ onto $\llbracket\mathsf{pr}_{a,b}\rrbracket$.
\end{proposition}

\begin{proof}
    It suffices to show the map $\langle r,s\rangle\mapsto\mathsf{mk}_{a,b}(r)(s)$ is invertible.  Specifically, we claim the inverse map is given by $p\times q$ where
    \begin{align*}
        p&=\llbracket\mathsf{rec}_{a,b}\rrbracket(\llbracket\lambda w\mathsf{pr}_{a,b}a\rrbracket)(\llbracket\lambda xa\lambda ybx\rrbracket)\quad\text{and}\\
        q&=\llbracket\mathsf{rec}_{a,b}\rrbracket(\llbracket\lambda w\mathsf{pr}_{a,b}b\rrbracket)(\llbracket\lambda xa\lambda yby\rrbracket).
    \end{align*}
    To see this, take any $r\in\llbracket a\rrbracket$ and $s\in\llbracket b\rrbracket$.  We can then change the interpretation $\llbracket\cdot\rrbracket$ so that $\llbracket x\rrbracket=r$ and $\llbracket y\rrbracket=s$.  Further changing it so that $\llbracket f\rrbracket=\llbracket\lambda w\mathsf{pr}_{a,b}a\rrbracket$ and $\llbracket g\rrbracket=\llbracket\lambda ya\lambda zby\rrbracket$ (so $\llbracket f\rrbracket(t)=\llbracket a\rrbracket$, for all $t\in\llbracket\mathsf{pr}_{a,b}\rrbracket$, and $\llbracket g\rrbracket(i)(j)=i$, for all $i\in\llbracket a\rrbracket$ and $j\in\llbracket b\rrbracket$), we see that
    \begin{align*}
        p(\mathsf{mk}_{a,b}(r)(s))&=\llbracket\mathsf{rec}_{a,b}\rrbracket(\llbracket\lambda w\mathsf{pr}_{a,b}a\rrbracket)(\llbracket\lambda ya\lambda zby\rrbracket)(\mathsf{mk}_{a,b}(r)(s))\\
        &=\llbracket\mathsf{rec}_{a,b}fg(\mathsf{mk}_{a,b}xy)\rrbracket\\
        &=\llbracket gxy\rrbracket=\llbracket x\rrbracket=r.
    \end{align*}
    Likewise, instead changing $\llbracket\cdot\rrbracket$ so $\llbracket f\rrbracket=\llbracket\lambda w\mathsf{pr}_{a,b}b\rrbracket$ and $\llbracket g\rrbracket=\llbracket\lambda ya\lambda zbz\rrbracket$, the same argument yields $q(\mathsf{mk}_{a,b}(r)(s))=s$.  This shows that $p\times q$ is a left inverse to the map $\langle r,s\rangle\mapsto\mathsf{mk}_{a,b}(r)(s)$.

    To show that $p\times q$ is also a right inverse, it suffices to show that $\llbracket\mathsf{mk}_{a,b}\rrbracket$ is surjective.  If not, then we could take $t\in\llbracket\mathsf{pr}_{a,b}\rrbracket\setminus\mathrm{ran}(\llbracket\mathsf{mk}_{a,b}\rrbracket)$.  Then we could further define $\phi\in\llbracket\mathsf{u}_n\rrbracket^{\llbracket\mathsf{pr}_{a,b}\rrbracket}$ so
    \[\phi(r)=\begin{cases}\{\emptyset\}&\text{if }r\in\mathrm{ran}(\llbracket\mathsf{mk}_{a,b}\rrbracket)\\\ \,\emptyset&\text{otherwise}.\end{cases}\]
    Taking $\theta\in\prod_{r\in\llbracket a\rrbracket}\prod_{s\in\llbracket b\rrbracket}(\phi(\llbracket\mathsf{mk}_{a,b}\rrbracket(r)(s)))=\prod_{r\in\llbracket a\rrbracket}\prod_{s\in\llbracket b\rrbracket}\{\emptyset\}$, necessarily with $\theta(r)(s)=\emptyset$, for all $r\in\llbracket a\rrbracket$ and $s\in\llbracket b\rrbracket$, we see that
    \[\llbracket\mathsf{rec}_{a,b}\rrbracket(\phi)(\theta)(t)\in\phi(t)=\emptyset,\]
    a contradiction.  Thus $\llbracket\mathsf{mk}_{a,b}\rrbracket$ must indeed be surjective.
\end{proof}

Note the proof of the last part above would have worked just as well if we had replaced $\mathsf{u}_n$ in the typing statement for $\mathsf{rec}_{a,b}$ with $\mathsf{u}_0$ or any other $\mathsf{u}_l$.  Indeed, when it comes to adding specifications like this as axioms for the syntactic system (see \Cref{vdashDef'} below), it makes sense to add corresponding statements $(\mathsf{rec}_{a,b}^l:(f:\mathsf{pr}_{a,b}\rightarrow\mathsf{u}_l)\rightarrow\ldots)$, for all $l\in\omega$.  In fact, even if we just wanted to extend the result above to models of $(a:\mathsf{u}_m)$ and $(b:\mathsf{u}_n)$ for distinct $m,n\geq1$, then we would need both $\mathsf{rec}_{a,b}^m$ and $\mathsf{rec}_{a,b}^n$ in the first part of the proof.

Again for the syntactic system, it is often better to add sub-reduction statements as axioms rather than reduction statements.  However, as long as the interpretation also satisfies the specification for equality on the appropriate set, the contextual closure of the corresponding sub-reduction statement still suffices to complete the specification, as indicated by \Cref{SubReductionEq}.  We can also replace the individual variables in these sub-reduction statements with arbitrary terms.

We can also make everything polymorphic by replacing $a$ and $b$ with $v$ and $w$ while adding $(v:\mathsf{u}_m)\rightarrow(w:\mathsf{u}_n)\rightarrow$ at the start of all the typing statements.  So the full polymorphic specification for binary Cartesian products would then consist of the following statements, for all $l,m,n\in\omega$ with $m,n\geq1$ and $F,G,V,W,X,Y\in\mathsf{T}$ --
\begin{align*}
    (\mathsf{pr}_{m,n}:\ &(v:\mathsf{u}_m)\rightarrow(w:\mathsf{u}_n)\rightarrow\mathsf{u}_{m\vee n}),\\
    (\mathsf{mk}_{m,n}:\ &(v:\mathsf{u}_m)\rightarrow(w:\mathsf{u}_n)\rightarrow v\rightarrow w\rightarrow\mathsf{pr}_{m,n}vw),\\
    (\mathsf{rec}_{m,n}^l:\ &(v:\mathsf{u}_m)\rightarrow(w:\mathsf{u}_n)\rightarrow(f:\mathsf{pr}_{m,n}vw\rightarrow\mathsf{u}_l)\\
    &\rightarrow((x:v)\rightarrow (y:w)\rightarrow f(\mathsf{mk}_{m,n}vwxy))\\
    &\rightarrow(z:\mathsf{pr}_{m,n}vw)\rightarrow fz)\quad\text{and}\\
        (\mathsf{rec}_{m,n}^lV&WFG(\mathsf{mk}_{m,n}VWXY)\triangleright GXY)^\mathsf{c}.
\end{align*}

The logical `and' operation on propositions can be specified in essentially the same way, just without the need for any sub-reduction statements.  Indeed, the proof of the following is the same as the last part of the proof of \Cref{CartesianProduct} above.

\begin{proposition}
    Given $f,g,v,w,x,y,z\in\mathsf{V}$ as well as constants $\mathsf{and},\mathsf{mk},\mathsf{rec}\in\underline{\mathsf{C}}$, if $\llbracket\cdot\rrbracket$ is a canonical model of
\begin{align*}
    (\mathsf{and}:\ &(v:\mathsf{u}_0)\rightarrow(w:\mathsf{u}_0)\rightarrow\mathsf{u}_0),\\
    (\mathsf{mk}:\ &(v:\mathsf{u}_0)\rightarrow(w:\mathsf{u}_0)\rightarrow v\rightarrow w\rightarrow\mathsf{and}vw)\quad\text{and}\\
    (\mathsf{rec}:\ &(v:\mathsf{u}_0)\rightarrow(w:\mathsf{u}_0)\rightarrow(f:\mathsf{and}vw\rightarrow\mathsf{u}_0)\\
    &\rightarrow((x:v)\rightarrow (y:w)\rightarrow f(\mathsf{mk}vwxy))\\
    &\rightarrow(z:\mathsf{and}vw)\rightarrow fz)
\end{align*}
then, for all $p,q\in\llbracket\mathsf{u}_0\rrbracket$, $\langle r,s\rangle\mapsto\mathsf{mk}(r)(s)$ maps $p\times q$ onto $\llbracket\mathsf{and}\rrbracket(p)(q)$.  In particular,
\[\llbracket\mathsf{and}\rrbracket(p)(q)\neq\emptyset\qquad\Leftrightarrow\qquad p\neq\emptyset\neq q.\]
\end{proposition}

Likewise, the complete specification of the universal quantifier for propositions over some set $\llbracket a\rrbracket$ can be seen from the following.

\begin{proposition}
    Given $n\in\omega$, $f,g,q,z\in\mathsf{V}$ and $a,\forall_a,\mathsf{mk}_a,\mathsf{rec}_a\in\underline{\mathsf{C}}$, if $\llbracket\cdot\rrbracket$ is a canonical model for $(a:\mathsf{u}_n)$, $(\forall_a:(q:a\rightarrow\mathsf{u}_0)\rightarrow\mathsf{u}_0)$,
\begin{align*}
    (\mathsf{mk}_a:\ &(q:a\rightarrow\mathsf{u}_0)\rightarrow\mathsf{p}aq\rightarrow\forall_aq)\quad\text{and}\\
    (\mathsf{rec}_a:\ &(q:a\rightarrow\mathsf{u}_0)\rightarrow(f:\forall_a q\rightarrow\mathsf{u}_0)\\
    &\rightarrow((g:\mathsf{p}aq)\rightarrow f(\mathsf{mk}_aqg))\rightarrow(z:\forall_aq)\rightarrow fz)\quad\text{then},
    \end{align*}
for all $\theta\in\llbracket\mathsf{u}_0\rrbracket^{\llbracket a\rrbracket}$, $\phi\mapsto\mathsf{mk}_a(\theta)(\phi)$ maps $\prod\theta$ onto $\llbracket\forall_a\rrbracket(\theta)$ and hence
\[\llbracket\forall_a\rrbracket(\theta)\neq\emptyset\qquad\Leftrightarrow\qquad \prod\theta\neq\emptyset.\]
\end{proposition}

Again the proof is just like the last part of the proof of \Cref{CartesianProduct} and the specification can be made polymorphic as above.

\section{Consequences}

We now define the semantic consequence relation $\vDash$ and examine its basic properties just like in \S\ref{Consequences}.

\begin{definition}
Define the \emph{consequence relation} ${\vDash}\subseteq\mathcal{P}(\mathsf{S})\times\mathsf{S}$ by
\[\Gamma\vDash X\qquad\Leftrightarrow\qquad\llbracket\Gamma\rrbracket\Rightarrow\llbracket X\rrbracket,\text{ for every interpretation }\llbracket\cdot\rrbracket.\]    
\end{definition}

So $\Gamma\vDash X$ means every model of $\Gamma$ is a model of $X$.

As before, we extend any ${\Vdash}\subseteq\mathcal{P}(\mathsf{S})\times\mathsf{S}$ to a binary relation on $\mathcal{P}(\mathsf{S})$ so that, for all $\Gamma,\Delta\subseteq\mathsf{S}$ and $S,P\in\mathsf{T}$,
\[\Gamma\Vdash\Delta\quad\Leftrightarrow\quad\Gamma\Vdash(S:P),\text{ for all }(S:P)\in\Delta.\]
We call ${\Vdash}$ a \emph{sequent} if this extension defines a preorder on $\mathcal{P}(\mathsf{S})$, which is again equivalent to saying that, for all $\Gamma,\Delta\subseteq\mathsf{S}$ and $X\in\mathsf{S}$,
\[\tag{Sequent}X\in\Gamma\text{ or }\Gamma\Vdash\Delta\Vdash X\qquad\Rightarrow\qquad\Gamma\Vdash X.\]

We immediately see that $\vDash$ is a sequent and hence also monotone and reflexive, exactly as in \Cref{SequentMonoRefl} and \Cref{ConMonoRefl}.  We also have the following property of $\vDash$ analogous to \Cref{AppProp}.

\begin{proposition}\label{AppConsequence2}
    For any $m,n\in\omega$ and $F,G,R,S\in\mathsf{T}$,
    \[(S:R),(F:\mathsf{p}_m^nRG)\vDash(FS:GS).\]
\end{proposition}

\begin{proof}
    Assume $\llbracket\cdot\rrbracket$ is a model for $(S:R)$ and $(F:\mathsf{p}_m^nRG)$ so $\llbracket S\rrbracket\in\llbracket R\rrbracket$ and $\llbracket F\rrbracket\in\prod\llbracket G\rrbracket$.  But also $\llbracket\mathsf{p}_m^nRG\rrbracket^\mathsf{wf}$ so $\llbracket R\rrbracket\in\llbracket\mathsf{u}_m\rrbracket$ and $\llbracket G\rrbracket\in\llbracket\mathsf{u}_n\rrbracket^{\llbracket R\rrbracket}$.  In particular, $\mathrm{dom}\llbracket F\rrbracket=\mathrm{dom}\llbracket G\rrbracket=\llbracket R\rrbracket\ni\llbracket S\rrbracket$ so $\llbracket FS\rrbracket^\mathsf{wf}$, $\llbracket GS\rrbracket^\mathsf{wf}$ and 
    \[\llbracket FS\rrbracket=\llbracket F\rrbracket(\llbracket S\rrbracket)\in\llbracket G\rrbracket(\llbracket S\rrbracket)=\llbracket GS\rrbracket.\]
    This shows that $\llbracket\cdot\rrbracket$ is a model of $(FS:GS)$, as required.
\end{proof}

The free variables in any set of statements $\Gamma\subseteq\mathsf{S}$ are given by
\[\mathsf{F}(\Gamma)=\bigcup_{(R:S)\in\Gamma}\mathsf{F}(RS).\]
So $\mathsf{F}(\Gamma)$ consists of the free variables of the subjects and predicates of typing statements in $\Gamma$.  Reduction and sub-reduction statements in $\Gamma$ have no bearing on $\mathsf{F}(\Gamma)$, owing to the fact models of (sub-)reduction statements are invariant under any change to the values of variables.

We now have the following analog of \Cref{AbProp}.

\begin{proposition}\label{AbProp2}
    If $x\notin\mathsf{F}(\Gamma)\cup\mathsf{F}(Q)$ and $\Gamma\vDash(Q:\mathsf{u}_m)$ then
    \[\Gamma,(x:Q)\vDash(S:P),(P:\mathsf{u}_n)\qquad\Rightarrow\qquad\Gamma\vDash(\lambda xQS:\pi_m^nxQP).\]
\end{proposition}

\begin{proof}
    Assume $\Gamma,(x:Q)\vDash(S:P),(S:\mathsf{u}_n)$ and take a model $\llbracket\cdot\rrbracket$ for $\Gamma$.  If $q\in\llbracket Q\rrbracket=\llbracket Q\rrbracket_{\langle q,x\rangle}$, as $x\notin\mathsf{F}(Q)$, then $\llbracket x:Q\rrbracket_{\langle q,x\rangle}$ and also $\llbracket\Gamma\rrbracket_{\langle q,x\rangle}$, as $x\notin\mathsf{F}(\Gamma)$ (and the fact that models of (sub-)reduction statements are invariant under changing the values of variables).  Thus $\llbracket\cdot\rrbracket_{\langle q,x\rangle}$ is a model for $\Gamma$ and $(x:Q)$ and hence for $(S:P)$ and $(P:\mathsf{u}_n)$ as well.  In particular, $\llbracket S\rrbracket_{\langle q,x\rangle}^\mathsf{wf}$, for all $q\in\llbracket Q\rrbracket$, and hence $\llbracket\lambda xQS\rrbracket^\mathsf{wf}$.  Likewise $\llbracket\lambda xQP\rrbracket^\mathsf{wf}$ and $\llbracket\lambda xQP\rrbracket\in\llbracket\mathsf{u}_n\rrbracket^{\llbracket Q\rrbracket}$.  As $\llbracket\Gamma\rrbracket$ and hence $\llbracket Q\rrbracket\in\llbracket\mathsf{u}_m\rrbracket$ too, it follows that $\llbracket\mathsf{p}_m^nQ\lambda xQP\rrbracket^\mathsf{wf}$.  By definition, $\pi_m^nxQP=\mathsf{p}_m^nQ\lambda xQP$ and, as in the proof of \Cref{AbProp},
    \[\hspace{-10pt}\llbracket\lambda xQS\rrbracket=\{\langle\llbracket S\rrbracket_{\langle q,x\rangle},q\rangle:q\in\llbracket Q\rrbracket\}\in\prod_{q\in\llbracket Q\rrbracket}\llbracket P\rrbracket_{\langle q,x\rangle}=\prod\llbracket \lambda xQP\rrbracket=\llbracket\pi_m^nxQP\rrbracket.\]
    This shows $\llbracket\lambda xQS:\pi_m^nxQP\rrbracket$ and hence $\Gamma\vDash(\lambda xQS:\pi_m^nxQP)$.
\end{proof}

We also observe that sub-reduction statements allow us to replace the predicates of typing statements as follows.

\begin{proposition}\label{PredicateReduction}
    For all $P,R,S\in\mathsf{T}$,
    \[(S:R),(R\triangleright P)\vDash(S:P).\]
\end{proposition}

\begin{proof}
    Just note that $\llbracket S:R\rrbracket$ implies $\llbracket S\rrbracket\in\llbracket R\rrbracket$ and $\llbracket R\rrbracket^\mathsf{wf}$ so $\llbracket R\triangleright P\rrbracket$ implies $\llbracket P\rrbracket^\mathsf{wf}$ and $\llbracket R\rrbracket\subseteq\llbracket P\rrbracket$ and hence $\llbracket S\rrbracket\in\llbracket P\rrbracket$, i.e.~$\llbracket S:P\rrbracket$.
\end{proof}

The same applies subjects under reduction.

\begin{proposition}\label{SubjectReduction}
    For all $P,R,S\in\mathsf{T}$,
    \[(R:P),(R\btright S)\vDash(S:P).\]
\end{proposition}

\begin{proof}
    As above, $\llbracket R:P\rrbracket$ implies $\llbracket R\rrbracket\in\llbracket P\rrbracket$ and $\llbracket R\rrbracket^\mathsf{wf}$ so $\llbracket R\btright S\rrbracket$ implies $\llbracket S\rrbracket^\mathsf{wf}$ and $\llbracket R\rrbracket=\llbracket S\rrbracket$ and hence $\llbracket S\rrbracket\in\llbracket P\rrbracket$, i.e.~$\llbracket S:P\rrbracket$.
\end{proof}

We now define our syntactic \emph{inference relation} ${\vdash}$ as the smallest sequent satisfying the rules arising from \Cref{AppConsequence2,AbProp2,PredicateReduction,SubjectReduction}.

\begin{definition}\label{vdashDef'}
    Let ${\vdash}\subseteq\mathcal{P}(\mathsf{S})\times\mathsf{S}$ be the smallest sequent such that, for all $F,G,P,Q,R,S\in\mathsf{T}$ and $m,n\in\omega$,
\begin{align}
    \label{Sub''}\tag{Red$_{\btright}$}(R:P),(R\mathbin{\btright}S)&\vdash(S:P),\\
    \label{Red''}\tag{Red$_\triangleright$}(S:R),(R\mathop{\triangleright}P)&\vdash(S:P),\\
    \label{App''}\tag{App}(S:R),(F:\mathsf{p}_m^nRG)&\vdash(FS:GS)\quad\text{and}\\
    \label{Ab''}\tag{Ab}\Gamma,(x:Q)\vdash(S:P),(P:\mathsf{u}_n)\quad\Rightarrow\quad\Gamma&\vdash(\lambda xQS:\pi_m^nxQP),
\end{align}
whenever $\Gamma\vdash(Q:\mathsf{u}_m)$ and $x\in\mathsf{V}\setminus(\mathsf{F}(\Gamma)\cup\mathsf{F}(Q))$.
\end{definition}

As the rules defining $\vdash$ we were already derived from the properties of $\vDash$ proved previously, the following result is immediate.

\begin{proposition}
    The inference relation is sound, i.e.~${\vdash}\subseteq{\vDash}$.
\end{proposition}

\section{Future Work}

So far we have set out the basic syntax and semantics of our simplified type system with polymorphic product operators.  But of course this is just the beginning and there would be much further work to do if this system were to be used as a foundation for mathematics.

Firstly, we should investigate basic properties of the inference relation $\vdash$ as in \S\ref{Inferences}.  The extra \eqref{Sub''} and \eqref{Red''} rules here complicate matters and accordingly we may want to restrict the reduction and sub-reduction statements that can be added as axioms.  In other words, results concerning instances $\Gamma\vdash(S:P)$ of the inference relation may only apply when $\Gamma_{\btright}:=\Gamma\cap\mathsf{S}_{\btright}$ and $\Gamma_\triangleright:=\Gamma\cap\mathsf{S}_\triangleright$ are of a specified form.

The most restrictive approach would be to only allow $\alpha$-conversion for the reduction statements while allowing both $\alpha$-conversion and $\beta$-reduction as sub-reduction statements but nothing else, much like in pure type systems.  Reduction statements in specifications like those of \Cref{CartesianProduct} would then have to be replaced by typing statements involving propositional equality.  A less restrictive approach would be to instead replace the reduction statements in such specifications by sub-reduction statements, as suggested in the comments after \Cref{CartesianProduct}.  This is essentially the approach taken in the dependent type theory underlying proof assistants like Lean -- see \cite{Carneiro2019,Carneiro2024}.  It is only when we are not at all concerned with type-checking being decidable that we would opt for an `extensional' system allowing arbitrary reduction statements, not just those for $\alpha$-conversion.

If we want unique typing as in \Cref{UniqueTyping} then we should also restrict the typing statements on the left side of $\vdash$ to \emph{contexts}, i.e. to $\Gamma\subseteq\mathsf{S}$ such that $\Gamma_:^{-1}$ is a function with $\mathrm{dom}(\Gamma_:^{-1})\subseteq\mathsf{T}_0$ (where $\Gamma_:=\mathsf{S}_:\cap\Gamma$).  If we also want to have any hope of proving some normalisation result (that the terms involved can only be reduced finitely many times by the \eqref{Sub''} and \eqref{Red''} rules, at least modulo $\alpha$-conversion) then we would have to further restrict to appropriate `legal' contexts.  Usually these are taken to be certain finite sequences of statements, but we could also define them for (even infinite) sets along the following lines.

\begin{definition}
    A context $\Gamma\subseteq\mathsf{S}$ is \emph{legal} if we have an enumeration $(x_n:P_n)_{n\in\omega}$ of $\Gamma_:$ such that, for all $j\in\omega$,
    \[x_j\notin\bigcup_{i\geq j}\mathsf{F}(P_i)\quad\text{and}\quad(\Gamma\setminus\bigcup_{i\leq j}(x_i:P_i))\vdash(P_j:\mathsf{u}_k),\text{ for some }k\in\omega.\]
\end{definition}

One more item on the to do list would be to verify the semantic completeness of various other specifications for important mathematical structures like in \Cref{Specifications}.  Even better would be to do this in a more systematic way.  Indeed, there are already systematic ways of defining inductive specifications from a given list of constructors -- see \cite[\S2.6]{Carneiro2019}.  These suffice to specify many mathematical structures and so a general result on their semantic completeness would cover most specifications of interest.

Finally, it would be nice to know if our inference relation satisfies any kind of completeness result to complement soundness, even for the bare bones type system in \Cref{Part1}.  In other words, are there any suitable conditions under which $\Gamma\vDash(S:P)$ implies $\Gamma\vdash(S:P)$ or even $\Gamma\vdash(T:P)$ for some other term $T$?  If not then what are the `missing' inference rules that would be required for this to hold?

\bibliography{References}{}
\bibliographystyle{alphaurl}

\end{document}